\documentclass[reqno,oneside]{amsart}
\usepackage{amsmath, amsthm}
\usepackage{amssymb}
\usepackage{stackengine}
\usepackage{xcolor}
\usepackage{appendix}
\usepackage[a4paper, total={6in, 9.5in}]{geometry}

\newtheorem{theorem}{Theorem}[section]
\newtheorem{lemma}[theorem]{Lemma}
\newtheorem{proposition}[theorem]{Proposition}
\theoremstyle{definition}
\newtheorem{definition}[theorem]{Definition}

\theoremstyle{remark}
\newtheorem{remark}[theorem]{Remark}

\numberwithin{equation}{section}

\def\Xint#1{\mathchoice
   {\XXint\displaystyle\textstyle{#1}}%
   {\XXint\textstyle\scriptstyle{#1}}%
   {\XXint\scriptstyle\scriptscriptstyle{#1}}%
   {\XXint\scriptscriptstyle\scriptscriptstyle{#1}}%
   \!\int}
\def\XXint#1#2#3{{\setbox0=\hbox{$#1{#2#3}{\int}$}
     \vcenter{\hbox{$#2#3$}}\kern-.5\wd0}}
\def\dashint{\Xint-}

\newcommand{\p}{\Delta_p}
\newcommand{\ps}{(-\Delta_p)^s}
\newcommand{\la} {\lambda}
\newcommand{\ba}{\beta}
\newcommand{\Om} {\Omega}
\newcommand{\pa}{\partial}
\newcommand{\sr}{\mathbb{R}}
\newcommand{\si}{\sigma}
\newcommand{\rar}{\rightarrow}
\newcommand{\na} {\nabla}
\newcommand{\wop}{W^{1,p}_0 (\Omega)}
\newcommand{\test}{C_c^{\infty} (\Om)}
\newcommand{\omc}{(\overline{\Om})}
\newcommand{\bdry}{\mathbb{R}^N \setminus \Omega}
\newcommand{\ftil}{\tilde{f}}

\newcommand{\De} {\Delta}
\newcommand{\sn}{\mathbb{N}}

\newcommand{\et}{\tilde{e}_0}

\begin{document}

\title[Multiplicity results for mixed local-nonlocal operators]{Multiple positive solutions to a perturbed Gelfand problem involving mixed local-nonlocal operators and singular nonlinearity}

\author{Sarbani Pramanik}
\email{sarbanipramanik20@iisertvm.ac.in}

\address{School of Mathematics, Indian Institute of Science Education and Research Thiruvananthapuram, Maruthamala, Thiruvananthapuram, Kerala, 695551, India.}



\begin{abstract}

We investigate a perturbed Gelfand problem involving a mixed local-nonlocal $p$-Laplacian operator with singular nonlinearity:
\begin{equation*}
\begin{aligned}
    -\Delta_p u + (-\Delta_p)^s u = \lambda \frac{f(u)}{u^{\beta}}\ \text{in} \ \Omega\\
    u >0\ \text{in} \ \Omega,\ u =0\ \text{in} \ \mathbb{R}^N \setminus \Omega
\end{aligned}
\end{equation*}
where $\Omega \subset \mathbb{R}^N$ is a smooth bounded domain, $\la > 0 $ is a parameter, $0\leq \beta<1$ and $f$ is a non-decreasing $C^1$-function with $f(0)>0$. Using the method of sub- and supersolutions, we present a novel multiplicity result and, in specific cases, we also prove a three-solution theorem using Amann's fixed point theorem. Our construction of sub-supersolutions avoids the conventional reliance on ODE techniques and Green's function estimates, thereby making it more adaptable to the nonlinear and nonlocal framework. Additionally, we establish a Hopf-type Strong Comparison Principle for the linear operator with singular nonlinearity, marking the first result of its kind for mixed local-nonlocal operators. This result is crucial in deriving a third solution and holds broader mathematical significance.

\end{abstract}

\maketitle


\section{Introduction}

Mixed local-nonlocal partial differential equations have attracted significant attention due to their relevance in modeling complex phenomena that exhibit both short-range and long-range interactions. These include, for example, anomalous diffusion processes, flame propagation, population dynamics, studying materials with heterogeneous structures, and even financial modeling. In such contexts, the classical $p$-Laplacian operator captures local diffusion effects, while the fractional $p$-Laplacian accounts for nonlocal interactions arising over spatially extended regions. The coexistence of these two operators in a single equation introduces rich mathematical structure, but also substantial analytical challenges -- especially when the nonlinearity involved is singular.

In this work, we investigate a boundary value problem involving a mixed local-nonlocal operator and a singular nonlinearity. Specifically, we consider the problem: 
\begin{equation}\tag{$P_{\la}$}\label{P_lambda}
    \begin{aligned}
        -\p u + \ps u = \la \frac{f(u)}{u^{\ba}} \text{ in } \Om\\
        u >0 \text{ in } \Om,\ u =0 \text{ in } \bdry
    \end{aligned}
\end{equation}
where $1<p<\infty$, $s\in(0,1)$, $0\leq \ba<1$, $\la>0$ is a parameter, the function $f:[0,\infty) \rar (0,\infty)$ is the nonlinearity and $\Om$ is a smooth bounded domain in $\sr^N$. The operator is a superposition of the local operator $p$-Laplacian, defined by $\p u:= div \left(|\na u|^{p-2} \na u\right)$, with the nonlocal fractional $p$-Laplacian, which is given by $$\ps u(x):= 2 P.V. \int_{\sr^N} \frac{\left|u(x) - u(y)\right|^{p-2} \left(u(x) - u(y)\right)}{|x-y|^{N+ps}}\ dy.$$ 
Our objective is to establish the existence and multiplicity of positive solutions to this problem \eqref{P_lambda}, with an emphasis on identifying parameter regimes where three distinct positive solutions can be guaranteed.

The interest in equation \eqref{P_lambda} is further motivated by its connection to combustion models in chemical kinetics. A prototypical example of the nonlinearity $f$ in this context is the function $f(t) = e^{\frac{\alpha t}{\alpha +t}}$, where $\alpha>0$ is large and denotes the activation energy parameter. In the non-singular case (i.e., $\beta =0)$, the corresponding local problem for the Laplacian becomes
\begin{equation}\label{perturbed_Gelfand}
    \begin{aligned}
        -\De u &= \la e^{\frac{\alpha u}{\alpha + u}} \text{ in } \Om\\
        u &=0 \text{ on } \pa\Om
    \end{aligned}
\end{equation}
which is regarded as a perturbed version of the classical Gelfand problem $-\De u = \la e^u$. Equation \eqref{perturbed_Gelfand} arises in the modeling of solid fuel combustion processes and specifically corresponds to the small fuel loss steady-state model \cite{bebernes2013mathematical}. In this setting, $u$ represents the temperature distribution within the medium while the exponential term describes the temperature-dependence governed by the Arrhenius reaction-rate law in irreversible chemical kinetics. The parameter $\la>0$ is the Frank-Kamenetskii parameter.

The problem described in equation \eqref{perturbed_Gelfand}, along with its various generalizations, has been the focus of a long-standing and rich line of research at the intersection of mathematical analysis and physical modeling -- see, for instance, \cite{bratu1914equations}, \cite{parks1961criticality}, \cite{gelfand1963}, \cite{joseph1972}, \cite{parter1974}, \cite{sattinger1975nonlinear}, \cite{tam1979}, \cite{dancer1980}, \cite{brown1981s}, \cite{du2001conjecture}, \cite{ramaswamy2004multiple}, \cite{hung2011jde}, \cite{huang2015bifurcation}, \cite{huang2016one}, and the references therein. Despite this substantial progress made over the years, the study of Gelfand-type problems continues to evolve, reflecting its enduring theoretical significance and practical importance. In more recent years, such problems have been extended in several directions, including the replacement of the Laplacian with nonlinear and nonlocal operators and the incorporation of singular terms. We refer to \cite{wang2008p}, \cite{rosoton2014fractional}, \cite{charro2023infinity}, \cite{huang2025minkowski}, \cite{molino2016singular} and \cite{molino2025random} among others, for relevant developments in these directions. Motivated by this dual perspective, the present work considers a perturbed Gelfand-type equation driven by a mixed local-nonlocal operator, and involves a possibly singular nonlinearity of the form $\frac{f(u)}{u^{\beta}}$ with $0\leq \beta<1$. This combination gives rise to new analytical challenges to establish the existence and multiplicity of positive solutions to \eqref{P_lambda}.

From a purely mathematical standpoint, the study of multiplicity results for nonlinear elliptic equations with singularities has a rich history, particularly in the case of local operators. Initial studies by \cite{brown1981s} and \cite{ramaswamy2004multiple} established the existence of three positive solutions for equations of the following type: 
 \[-\p u = \la f(u) \text{ in } \Omega ,\ u = 0 \text{ on } \pa\Om,\]
where $f$ is a non-decreasing $C^1$-function on $[0,\infty)$ that is $p$-sublinear at infinity and satisfies $f(0)>0$. In subsequent works, such as \cite{dhanya2015three} and \cite{ko2011multiplicity}, these results were extended to similar problems involving singular nonlinearities, that is, replacing the right-hand side of the above equation with $\la \frac{f(u)}{u^{\ba}}$ for some $\ba\in(0,1)$. More recently, multiplicity results for non-homogeneous operators (\cite{acharya2021existence}, \cite{arora2021multiplicity}) and Schr\"odinger operators with $L^{\infty}$-potentials \cite{ko2024multiplicity} have also been established. 

In contrast, the literature on multiplicity results for nonlocal or mixed local-nonlocal operators remains less developed. Existing results for mixed operators include \cite{garain2022mixed}, \cite{arora2023combined}, \cite{garain2023class}, \cite{biroud2023mixed}, \cite{su2024multiple} and \cite{biagi2024multiplicity} among others, where they establish the existence of one or two solutions via variational methods or monotone iteration techniques. However, results that establish the existence of three distinct solutions in the mixed setting are largely absent. In this work, we aim to build on the classical three-solution framework introduced in the seminal contributions \cite{brown1981s}, \cite{ramaswamy2004multiple}, and explore its extension to the setting of mixed local-nonlocal linear operators. A related contribution in the nonlocal case is given in \cite{giacomoni2019existence}, where the authors established the existence of three positive solutions for singular problems involving the fractional Laplacian. However, their approach relies significantly on the linearity of the operator and suitable estimates on the associated Green's function, limiting its adaptability to nonlinear operators.

In this article, we study the boundary value problem $(P_\lambda)$ under a set of structural conditions on the nonlinearity $f$, aimed at ensuring the existence of multiple positive solutions. The function $f:[0,\infty)\rightarrow (0,\infty)$ is assumed to satisfy the following conditions: 
\begin{itemize}
    \item[$(f1)$]  $f \in C^1 \left([0,\infty)\right)$ with $f(0)>0$,

    \item[$(f2)$] $f$ is non-decreasing in $\sr^+$,

    \item[$(f3)$] $\lim_{t\rar\infty} \frac{f(t)}{t^{p-1+\ba}} =0 $,

    \item[$(f4)$] there exist constants $\si_1, \si_2$ such that $0< \si_1 < b\si_2 <\si_2$ and $\frac{f(t)}{t^{\ba}}$ is non-decreasing on $(\si_1 , \si_2)$, where $0<b=b(p,s,N,\Om)<1$ is given in \eqref{b_B},

    \item[$(f5)$] there exists a constant $a\in (\sigma_1,b\sigma_2)$ such that
    \[Q(\si_1 , a):= \frac{\si_1^{p-1+\ba}}{f(\si_1)} \frac{f(a)}{a^{p-1+\ba}} >B,\]
    where $B=B(\ba,p,s,N,\Om)>0$ is a fixed constant explicitly defined in \eqref{b_B}.
\end{itemize}

While the assumptions (f4) and (f5), and consequently the parameter regime in which the multiplicity result holds (see Theorem \ref{theorem_1} and the associated expressions for $\la_*$ and $\la^*$ given in remark \ref{region_multiplicity}) may depend on the geometry of the domain $\Om$, this dependence can be relaxed in certain special cases. For instance, in the case of the model nonlinearity $f(t) = e^{\frac{\alpha t}{\alpha +t}}$ -- as observed in the work of Brown-Ibrahim-Shivaji (\cite{brown1981s}, p. 485) for Laplacian -- even in our setting the multiplicity of positive solutions persists for arbitrary domains $\Om$, provided the parameter $\alpha$ is sufficiently large. This observation also holds a clear physical relevance. Indeed, taking large values of $\alpha$ corresponds to the regime of large activation energy asymptotics, where the reaction rate becomes sharply sensitive to the temperature -- a key feature in thermal ignition and combustion theory.

The novelty of our work lies in providing, to the best of our knowledge, the first multiplicity and three-solution results for mixed local-nonlocal problems with singular nonlinearity. A key component of our approach is the construction of suitable sub- and supersolutions using auxiliary functions tailored to the mixed setting. Unlike traditional methods that rely on ODE techniques -- commonly applied to the $p$-Laplacian (see, for example, \cite{ramaswamy2004multiple}, \cite{ko2011multiplicity}) or $(p,q)$-Laplacian (see \cite{acharya2021existence}, \cite{arora2021multiplicity}) -- and estimates on Green's function, often used in the setting of fractional Laplacian (see \cite{giacomoni2019existence}), our approach follows a different path. This alternate framework offers greater flexibility in carrying out the analysis for mixed nonlinear operators.

Another significant aspect of this work is the development of a Hopf-type Strong Comparison Principle for linear mixed operators with singular nonlinearity (Theorem \ref{Hopf lemma_linear case}). This result is a crucial tool in deriving a third positive solution when two positive solutions are already known. Notably, no such result has been established in the literature for singular problems with mixed operators. In this paper, in the case of $p=2$, we prove a Hopf-type Strong Comparison result (Theorem \ref{Hopf lemma_linear case}), which not only supports our primary objective but also stands as a significant contribution in its own right. Additionally, our result applies for every $s\in (0,1)$ and thus improves the Strong Comparison Principle (Theorem 2.8) of \cite{dhanya2024interiorboundaryregularitymixed}.

Now, we summarize the main findings of this work in the following theorems.

\begin{theorem}\label{theorem_1}
    Let $f$ satisfy $(f1)$-$(f3)$. Then
    \begin{enumerate}
        \item[(a)] the problem \eqref{P_lambda} admits a positive solution for every $\la>0$,

        \item[(b)] if, additionally, $f$ satisfies $(f4)$-$(f5)$, then there exist $0<\la_* < \la^*$ such that for every $\la\in (\la_* , \la^*)$, the problem \eqref{P_lambda} admits two positive solutions.
    \end{enumerate}
\end{theorem}

This result can be further improved for two special cases: the non-singular scenario (i.e., $\ba=0$ and $p>1$) and the linear singular case (i.e., $p=2$ and $0<\ba <1$). In both settings, we establish the existence of a third solution that is distinct from the two obtained in Theorem \ref{theorem_1}. This indicates that the bifurcation diagram possesses at least two turning points, giving rise to an $S$-like structure. In case of an exact $S$-shaped bifurcation curve, the two critical values associated with the turning points correspond to the ignition limit and the extinction limit, respectively. The upper branch of the curve is commonly referred to as the explosion branch, while the lower branch is called the quenching branch. For details, we refer to \cite{kapila1979}.

\begin{theorem}\label{theorem_2}
    Suppose $p>1$, $\ba=0$ and $f$ satisfies $(f1)$-$(f5)$. Then for every $\la\in (\la_* , \la^*)$, the problem \eqref{P_lambda} admits three positive solutions.
\end{theorem}

As noted earlier, the Hopf-type Strong Comparison Principle serves as a key ingredient in obtaining the third solution in both singular and non-singular cases. Unlike the non-singular case, however, one must first establish this principle for singular problems in the linear setting before proceeding to prove the existence of a third solution for \eqref{P_lambda}.

\begin{theorem}[\textbf{Strong Comparison Principle}]\label{Hopf lemma_linear case}
    Let $u_1, u_2 \in C^{1,\alpha}(\overline{\Om})$, for some $\alpha\in(0,1)$, $u_1, u_2 > 0$ in $\Om$ satisfy 
    \begin{equation}\label{scp_equns}
        \begin{aligned}
            -\Delta u_1 + (-\De)^s u_1 - \frac{1}{u_1^{\ba}} = f_1 \text{ in } \Om\\
            -\Delta u_2 + (-\De)^s u_2 - \frac{1}{u_2^{\ba}} = f_2 \text{ in } \Om\\
            u_1 =u_2 =0 \text{ in } \bdry
        \end{aligned}
    \end{equation}
    where $f_1,f_2$ are continuous functions with $0 \leq f_1 \leq f_2$, $f_1 \not\equiv 0$ and $f_1 \not\equiv f_2$ in $\Om$. Then $u_1 < u_2$ in $\Om$ and $\frac{\pa u_2}{\pa\nu} < \frac{\pa u_1}{\pa\nu} < 0$ on $\pa\Om$.
\end{theorem}

\begin{theorem}\label{theorem_3}
    Suppose $p=2$, $0<\ba<1$ and $f$ satisfies $(f1)$-$(f5)$. Then for every $\la\in (\la_* , \la^*)$, the problem \eqref{P_lambda} admits three positive solutions.
\end{theorem}

The proofs of these theorems are based on the method of sub- and supersolutions. In this method, the central idea is that given a subsolution $\phi$ and a supersolution $\psi$ with $\phi \leq \psi$, we anticipate the existence of a solution $z$ such that $\phi \leq z\leq \psi$. Extending this approach, the three-solution theorem asserts that if we have two pairs $(\phi_1,\psi_1)$ and $(\phi_2,\psi_2)$ of sub- and supersolutions satisfying $\phi_1 \leq \phi_2 \leq \psi_1$, $\phi_1 \leq \psi_2 \leq \psi_1$, $\phi_2 \not\leq \psi_2$ and $\phi_2$, $\psi_2$ are respectively strict sub- and supersolutions, then there exist three distinct solutions $z_1 \in \left[\phi_1, \psi_2\right]$, $z_2 \in \left[\phi_2, \psi_1\right]$ and $z_3 \in \left[\phi_1,\psi_1\right]\setminus \left( \left[\phi_1,\psi_2\right] \cup \left[\phi_2,\psi_1\right] \right)$, where $[\phi_1, \psi_1]$ and similar terms denote the respective ordered intervals.

The paper is organized as follows: Section \ref{section-preliminaries} provides a brief overview of the preliminary results. In Section \ref{section-construction}, we construct two pairs of sub- and supersolutions with the required properties. Section \ref{section-multiplicity} focuses on the multiplicity result stated in Theorem \ref{theorem_1}. Next, in Section \ref{section-third solution}, we establish the existence of a third positive solution (Theorem \ref{theorem_2}) in the non-singular case ($1<p<\infty$ and $\ba =0$). In Section \ref{section-third solution- linear singular case}, we first prove a Strong Comparison Principle (Theorem \ref{Hopf lemma_linear case}) for singular problems in the linear case ($p=2$ and $0<\beta<1$) and then establish Theorem \ref{theorem_3}.


\section{Preliminaries}\label{section-preliminaries}

We begin this section by laying out the functional framework and defining notations that will be used consistently throughout the paper. Let $\Om \subset \sr^N$ be a bounded domain with smooth boundary $\pa\Om$ and fix $1<p<\infty$ and $s\in(0,1)$. We recall the standard Sobolev space $W^{1,p}(\Om)$, which is a Banach space equipped with the norm
\[\|u\|_{W^{1,p}(\Om)}:= \left(\|u\|^p_{L^p(\Om)} + \|\na u\|^p_{L^p(\Om)} \right)^{1/p}.\]
The space $\wop$ is defined as the closure of $\test$ with respect to the above norm. Next, the fractional Sobolev space $W^{s,p}(\sr^N)$ is defined as
\[W^{s,p}(\sr^N):= \left\{ u\in L^p(\sr^N): [u]_{s,p} < \infty \right\}, \]
where the term $[u]_{s,p}$ is the Gagliardo seminorm, given by
\[[u]_{s,p} = \left( \int_{\sr^N \times \sr^N} \frac{|u(x)-u(y)|^p}{|x-y|^{N+ps}}\ dx\ dy \right)^{1/p}.\]
The vector space space $W^{s,p}(\sr^N)$ is also a Banach space with the norm
\[\|u\|_{W^{s,p}(\sr^N)} = \left(\|u\|^p_{L^p(\sr^N)} + [u]^p_{s,p} \right)^{1/p}.\]

Now, to address both local and nonlocal operators within a unified framework, we introduce the function space $\mathcal{X}^{1,p}_0(\Om)$, defined as the completion of $\test$ with respect to the norm
\[\|u\|_{\mathcal{X}^{1,p}_0(\Om)}:= \left( \|u\|^p_{W^{1,p}(\sr^N)} + [u]^p_{s,p} \right)^{1/p},\ u\in\test.\]
This definition ensures that if $u\in \mathcal{X}^{1,p}_0(\Om)$, then $u=0$ a.e. in $\bdry$, thereby naturally conforming to the Dirichlet boundary condition.

Since the embedding $W^{1,p}(\sr^N) \hookrightarrow W^{s,p}(\sr^N)$ holds, there exists a constant $C=C(N,p,s,\Om)$ such that
\[[u]^p_{s,p} \leq C \|u\|^p_{W^{1,p}(\sr^N)} = C \left(\|u\|^p_{L^p(\Om)} + \|\na u\|^p_{L^p(\Om)} \right)\]
for every $u\in\test$. Combined with the Poincar\'e inequality, this implies that the norm $\|\cdot\|_{\mathcal{X}^{1,p}_0(\Om)}$ is equivalent to the gradient norm $\|u\|= \|\na u\|_{L^p(\Om)}$ on $\test$.

Further, due to the smoothness of $\pa\Om$, the space $\mathcal{X}^{1,p}_0(\Om)$ can be identified with the space $\wop$. Specifically, we shall interpret any $u\in\wop$ as its zero extension $\tilde{u}:= u\cdot \mathbf{1}_{\Om} \in \mathcal{X}^{1,p}_0(\Om)$, where $\mathbf{1}_{\Om}$ is the indicator function of $\Om$. Throughout the paper, we will make use of this tacit identification without additional remark.

\begin{definition}
    We say $u\in\wop$ is a solution to the problem \eqref{P_lambda} if $ess\ inf_K$ $u >0$ for every $K\Subset\Om$ and for every $\varphi\in\test$,
    \begin{equation*}
        \hspace{-20pt}\int_{\Om} |\na u|^{p-2} \na u \cdot \na\varphi\ dx + \int_{\sr^N \times \sr^N} \hspace{-12pt} \frac{|u(x) - u(y)|^{p-2} \left(u(x)-u(y)\right) \left(\varphi (x) - \varphi (y)\right)}{|x-y|^{N+ps}}\ dx\ dy = \la \int_{\Om} \frac{f(u)}{u^{\ba}} \varphi\ dx.
    \end{equation*}
    Similarly, we define a function $u\in\wop$ as a subsolution or a supersolution to \eqref{P_lambda} if $ess\ inf_K$ $u >0$ for every $K\Subset\Om$ and the above identity holds for every non-negative $\varphi\in\test$ with the equality sign being replaced by $\leq$ or $\geq$ respectively.
\end{definition}

To facilitate the analysis, we introduce the following function spaces. For any fixed function $\varphi \in C_0 \omc$ with $\varphi > 0$ in $\Om$, define
\[C_{\varphi} (\overline{\Om}):= \left\{u\in C_0 \omc : \exists\ c>0 \text{ such that } |u(x)|\leq c\varphi (x) \text{ in } \Om \right\}\]
\[ \text{and } C_{\varphi}^+ (\overline{\Om}):= \left\{u\in C_{\varphi} \omc : \inf_{x\in\Om} \frac{u(x)}{\varphi (x)} >0 \right\}. \]
We denote the distance function by $d(x)$, defined as $d(x)= dist (x,\pa\Om)$ for $x\in\Om$. For two functions $f_1, f_2 \in C\omc$, we write $f_1\sim f_2$ to express that there exist two positive constants $c_1, c_2$ such that $c_1 f_2(x) \leq f_1(x) \leq c_2 f_2(x)$ for all $x\in\Om$.

Next, we introduce two functions $\phi_0$ and $\psi_0$ that will be instrumental in the construction of sub- and supersolutions of \eqref{P_lambda} in Section \ref{section-construction}.
Since $\Om$ is a bounded domain, let $x_0$ be a point in $\Om$ and $R_0$ be a positive number such that $B_{R_0} (x_0)$ is the largest ball inscribed in $\Om$. Fix $0<R<R_0$ and suppose $\phi_0$ is the unique positive solution of the following PDE
\begin{equation}\label{phi_0}
    \begin{aligned}
        -\p \phi_0 + \ps \phi_0 &= \chi_{B_R} \text{ in } \Om\\
        \phi_0 &= 0 \text{ in } \bdry
    \end{aligned}
\end{equation}
where $B_R$ is the ball of radius $R$ around the point $x_0$. Then, by Theorem 1.1 and 1.2 of \cite{antonini2025global}, $\phi_0 \in C^{1,\alpha} \omc$, for some $\alpha\in (0,1)$ and $\phi_0 (x) \sim d(x)$ in $\Om$. Let $C_0, C_1$ be positive constants such that
\begin{equation}\label{C_0_C_1}
    C_0 d(x) \leq \phi_0 (x) \leq C_1 d(x), \forall\ x\in\Om.
\end{equation}

Now, consider another function $\psi_0$ which is the unique solution of
\begin{equation}\label{psi_0}
    \begin{aligned}
        -\p \psi_0 + \ps \psi_0 = \frac{2}{\psi_0^{\ba}} \text{ in } \Om\\
        \psi_0 >0 \text{ in } \Om,\ \psi_0 = 0 \text{ in } \bdry.
    \end{aligned}
\end{equation}
For the $\ba=0$ case, Theorem 1.1 and 1.2 of \cite{antonini2025global} ensures that $\psi_0 \in C^{1,\alpha} \omc$ for some $\alpha\in (0,1)$ and $\psi_0 (x) \sim d(x)$ in $\Om$. In the singular case, i.e. when $0<\ba<1$, the regularity results of \cite{dhanya2024interiorboundaryregularitymixed} guarantee that $\psi_0 \in C^{1,\alpha} \omc$ and $\psi_0 (x) \sim d(x)$ in $\Om$.

With all the necessary elements introduced, we are now in a position to provide explicit expressions for the constants $b$ and $B$ appearing in $(f4)$ and $(f5)$ respectively. We define
\begin{equation}\label{b_B}
    b:=\frac{2(R_0 - R)C_0}{(diam\ \Om) C_1} <1 \text{ and } B:= \frac{2\|\psi_0\|_{\infty}^{p-1+\ba}}{C_0 (R_0 -R)^{p-1}},
\end{equation}
where $diam\ \Om$ denotes the diameter of $\Om$ and $C_0, C_1$ are given by \eqref{C_0_C_1}.

Finally, throughout the paper, we use $C$ to represent arbitrary positive constants, unless specified otherwise.


\section{Construction of Sub-Supersolutions}\label{section-construction}

Let $e$ be the unique solution of 
\begin{equation}\label{e}
    \begin{aligned}
        -\p e + \ps e &= 1 \text{ in } \Om\\
        e &= 0 \text{ in } \bdry.
    \end{aligned}
\end{equation}
Then by the global regularity result and Hopf lemma of \cite{antonini2025global}, $e\in C^{1,\alpha} \omc$, $e>0$ in $\Om$ and $\frac{\pa e}{\pa \nu}<0$ on $\pa\Om$.

\begin{lemma}\label{first_pair}
    Assume $(f1)$-$(f3)$. Then for every $\la>0$, there exist positive numbers $m \ll 1$ and $M\gg 1$, both depending on $\la$, such that the functions $\phi_1:= me$ and $\psi_1:= M\psi_0$ serve respectively as strict sub- and supersolutions of \eqref{P_lambda}.
\end{lemma}

\begin{proof}
    Let $\la>0$ be fixed. Since $f(0)>0,$ we have $\lim_{t\rightarrow 0^+}\frac{f(t)}{t^{\ba}} \rar \infty.$ Hence, we can choose a sufficiently small constant $m\in (0,1)$ such that $\frac{2}{\la} \leq \frac{f(me)}{(me)^{\ba}}.$ Define $\phi_1:= me.$ Using the homogeneity of the operator and the choice of $m$, we obtain 
    \begin{equation}\label{sub_soln_1}
        -\p \phi_1 + \ps \phi_1 = m^{p-1} \,\leq \, 1 \,  \leq \frac{\la}{2} \frac{f(me)}{(me)^{\ba}} \; < \, \la \frac{f(\phi_1)}{\phi_1^{\ba}} \text{ in } \Om.
    \end{equation}
    Therefore, $\phi_1$ is a strict subsolution of \eqref{P_lambda}.
    
    To construct a supersolution, we first observe from the condition $(f3)$ that $\lim_{t\rightarrow \infty} \frac{f(t)}{t^{p-1+\beta}}=0.$ Therefore, we can choose an $M\gg 1$ such that
    \[\frac{f\left(M \|\psi_0\|_{\infty} \right)}{\left(M \|\psi_0\|_{\infty} \right)^{p-1+\ba}} \leq \frac{1}{\la \|\psi_0\|_{\infty}^{p-1+\ba}}.\]
    Define $\psi_1:=M \psi_0.$ Now using the homogeneity of the operator, the monotonicity of $f$ and the construction of $M$, we obtain
    \begin{equation*}
        -\p \psi_1 + \ps \psi_1 = \frac{2 M^{p-1}}{\psi_0^{\ba}} > \frac{M^{p-1}}{\psi_0^{\ba}} = \frac{M^{p-1+\ba}}{\psi_1^{\ba}} \geq \la \frac{f\left(M \|\psi_0\|_{\infty}\right)}{\psi_1^{\ba}} \geq \la \frac{f\left(\psi_1\right)}{\psi_1^{\ba}} \text{ in } \Om.
    \end{equation*}
    This proves that $\psi_1$ is a strict supersolution of \eqref{P_lambda}.
\end{proof}
    
\begin{remark}
    If $\la\in (\underline{\la}, \overline{\la})$ for some fixed $\underline{\la}>0$ and $\overline{\la}<\infty,$  then the constants $m$ and $M$ can be chosen uniformly for all $\lambda$ in this interval. In other words, $m$ and $M$ then depend only on $\underline{\la}$ and $\overline{\la}.$
\end{remark}

In the next two lemmas, we build a second pair of sub-supersolutions for the problem \eqref{P_lambda}.

\begin{lemma}\label{second_supersolution}
    Assume $(f1)$-$(f5)$. Then there exists $\la_1 >0$ such that for every $\la < \la_1$, the problem \eqref{P_lambda} admits a second strict supersolution $\psi_2$.
\end{lemma}

\begin{proof}
    Define $\psi_2:= \frac{\si_1 \psi_0}{\|\psi_0\|_{\infty}} $, where $\psi_0$ is given by \eqref{psi_0} and $\sigma_1$ is the constant given in $(f4)$. Set
    \begin{equation}\label{lambda_3}
        \la_1:= \frac{\si_1^{p-1+\ba}}{f(\si_1) \|\psi_0\|_{\infty}^{p-1+\ba}}.
    \end{equation}
    Using the monotonicity of $f,$ we deduce the following holds in $\Om$ for all $\la<\lambda_1$
    \begin{equation*}
        -\p \psi_2 + \ps \psi_2 = \frac{2 \si_1^{p-1}}{\psi_0^{\ba}\|\psi_0\|_{\infty}^{p-1}} > \frac{1}{\psi_0^{\ba}} \left( \frac{\si_1}{\|\psi_0\|_{\infty}} \right)^{p-1} = \frac{1}{\psi_2^{\ba}} \left( \frac{\si_1}{\|\psi_0\|_{\infty}} \right)^{p-1 +\ba} \hspace{-8pt} > \frac{\la f(\si_1)}{\psi_2^{\ba}} \geq \frac{\la f(\psi_2)}{\psi_2^{\ba}}.
    \end{equation*}
    Thus $\psi_2$ is a strict supersolution of \eqref{P_lambda} for every $\la < \la_1$.
\end{proof}

Next, we introduce a cut-off function $h$ which will facilitate the construction of the second subsolution. Recall that $\frac{f(t)}{t^{\ba}} \rar \infty \text{ as } t\rar 0+$ and let $\sigma_1,\sigma_2$ be as given in assumption $(f4).$ There exists a point $\si_* \in (0, \si_1]$ such that $\inf_{0<t\leq \si_1} \frac{f(t)}{t^{\ba}}$ is attained at the point $\si_*$, that is,
$$\inf_{0<t\leq \si_1} \frac{f(t)}{t^{\ba}} = \frac{f(\sigma_*)}{\sigma_*^{\ba}}.$$
Now, we define a $C^1$ function $h: [0,\infty) \rar \sr^+$ satisfying
\begin{equation}
    h(t)=\begin{cases}
        \frac{f(\si_*)}{\si_*^{\ba}}, \text{ for } t\leq \si_* \vspace{3mm}\\
        \frac{f(t)}{t^{\ba}}, \text{ for } t\geq \si_1
        \end{cases}
\end{equation}
and extend $h$ on the interval $(\si_*, \si_1)$ so that the following properties hold:
\begin{enumerate}
    \item[$(h1)$] $h$ is non-decreasing on $(0,\si_2)$,
    
    \item[$(h2)$] $h(t)\leq \frac{f(t)}{t^{\ba}}$ for all $t> 0.$
\end{enumerate}

\begin{lemma}\label{second_subsolution}
    Assume $(f1)$-$(f5)$. Then there exist $0<\la_2 <\la_3$ such that for every $\la\in (\la_2, \la_3)$, the problem \eqref{P_lambda} admits a second strict subsolution $\phi_2$ satisfying $\phi_2 \not\leq \psi_2$ in $\Om$.
\end{lemma}

\begin{proof}
   Let $\phi_2$ be the solution of the following auxiliary problem:
   \begin{equation*}
       \begin{aligned}
           -\p \phi_2 + \ps \phi_2 = \frac{\la}{2} h(a) \chi_{B_R} \text{ in } \Om\\
           \phi_2 > 0 \text{ in } \Om,\ \phi_2 =0 \text{ in } \bdry
       \end{aligned}
   \end{equation*}
   where $a$ is given in $(f5)$ and $B_R$ as introduced in equation \eqref{phi_0}, Section \ref{section-preliminaries}. Our aim is to show that $\phi_2$ is a strict subsolution of \eqref{P_lambda} for all $\la\in (\lambda_2,\lambda_3)$ where
   \begin{equation}\label{lambda_2}
       \lambda_2:= \frac{2a^{p-1}}{C_0^{p-1} (R_0 - R)^{p-1} h(a)}\;\;\;\;\mbox{and} \;\;\;\; \lambda_3:= \frac{2\si_2^{p-1}}{C_1^{p-1} \left( \frac{diam\ \Om}{2}\right)^{p-1} h(a)}.
   \end{equation}
   We first establish the following:
   
   \noindent\textbf{Claim:} For every $ \lambda\in (\la_2,\la_3)$ the function $\phi_2$ satisfies $\phi_2(x)\geq a$ for all $x\in B_R$ and $0<\phi_2(x)\leq \sigma_2$ for all $x\in \Omega$.
   
   By the homogeneity of the operator, it is easy to verify that
   \begin{equation}\label{est_second_subsolution}
       C_0 \left(\frac{\la}{2} h(a) \right)^{\frac{1}{p-1}} d(x) \leq \phi_2 (x) \leq C_1 \left(\frac{\la}{2} h(a) \right)^{\frac{1}{p-1}} d(x),\ \forall\ x\in\Om,
   \end{equation}
   where $C_0,C_1$ are given in \eqref{C_0_C_1} and $d(x)$ denotes distance of $x$ to the boundary of $\Omega$. 
   
   If $x\in B_R,$ then $d(x)\geq (R-R_0).$ Consequently, for $\la>\la_2$ it follows that $\phi_2 \geq a$ in $B_R.$ Similarly, if $\la<\la_3,$ using the properties of distance function we have $$ \phi_2(x)\leq C_1 \left(\frac{\la}{2} h(a) \right)^{\frac{1}{p-1}} \left(\frac{diam\ \Om}{2}\right) < \si_2 .$$ This verifies the claim.
    
    Further, observe that the condition $a< b\si_2$ in $(f5)$ ensures that $\la_2 < \la_3.$ Thus, for every 
    $\la\in(\la_2,\la_3)$ we now have 
    $$\|\phi_2\|_{L^\infty(\Omega)} \leq \si_2 \mbox{ and }\phi_2 \geq a \mbox{ on } B_R.$$
    Using the properties $(h1)$ and $(h2)$ of $h$, for every $\la\in(\la_2,\la_3)$ we obtain 
    \begin{equation*}
        -\p \phi_2 + \ps \phi_2= \frac{\la}{2} h(a)\chi_{B_R} \leq \frac{\la}{2} h(\phi_2) \leq \frac{\la}{2} \frac{f(\phi_2)}{\phi_2^{\ba}} < \la \frac{f(\phi_2)}{\phi_2^{\ba}} \text{ in } \Om.
    \end{equation*}
    This implies $\phi_2$ is a strict sub-solution of $(P_\la)$ whenever $\la\in(\la_2,\la_3).$ Finally, we note that $\|\psi_2\|_{\infty} = \si_1$, $\phi_2 \geq a$ on $B_R$ and $\si_1 < a$ together confirm that $\phi_2 \not\leq \psi_2$ in $\Om$. This completes the proof.
\end{proof}
    
\begin{remark}\label{region_multiplicity}
    Note that the condition $Q(\si_1,a)>B$ in $(f5)$ implies that $\la_2 < \la_1$. Therefore, we define $\la_* := \la_2$ and $\la^* := \min \{ \la_1, \la_3\}$ so that the open interval $(\la_*, \la^*)$ is non-empty and for every $\la\in(\la_*, \la^*)$, there exist two pairs $(\phi_1,\psi_1)$ and $(\phi_2,\psi_2)$  of (strict) sub- and supersolutions of \eqref{P_lambda} with $\phi_2 \not\leq \psi_2$ in $\Om$.
\end{remark}

\begin{remark}
    In lemma \ref{first_pair}, observe that if $me$ is a subsolution of \eqref{P_lambda}, so is the function $\tilde{m} e$ for any $0<\tilde{m} \leq m$. Similarly, if $M\psi_0$ is a supersolution of \eqref{P_lambda}, so is the function $\tilde{M} \psi_0$ for any $\tilde{M} \geq M$. In other words, we can choose $m>0$ small enough and $M$ large enough so that $\phi_1=me$ and $\psi_1 = M\psi_0$ are respectively sub- and supersolutions of \eqref{P_lambda} satisfying $\phi_1 \leq \phi_2 \leq \psi_1$ and $\phi_1 \leq \psi_2 \leq \psi_1$ as well.
\end{remark}


\section{Existence of Two Positive Solutions}\label{section-multiplicity}

We now proceed to the proof of Theorem \ref{theorem_1}. To handle the singular case effectively, we adapt the approach from \cite{dhanya2015three}, rewriting the equation by absorbing the singular term into the operator. The reformulated equation is the following:
\begin{equation}\tag{$\tilde{P}_{\la}$}\label{P_lambda_tilde}
    \begin{aligned}
        -\p u + \ps u &-\la \frac{f(0)}{u^{\ba}} = \tilde{f}(u) \text{ in } \Om\\
        u >0 \text{ in } \Om,\ u &=0 \text{ in } \bdry
    \end{aligned}
\end{equation}
where $\ftil (t) = \la \left(\frac{f(t)-f(0)}{t^{\ba}}\right)$, $t>0$. Clearly, $\ftil$ is a continuous function on $[0,\infty)$ with $\ftil (0)=0$, as a consequence of $f\in C^1([0,\infty))$ and the mean value theorem. We also assume that
\begin{itemize}
    \item[$(h)$] there exists $\tilde{k}>0$ such that the map $t \mapsto \ftil (t) +\tilde{k} t$ is increasing in $[0,\infty)$.
\end{itemize}
Therefore, without loss of generality, we can consider $\ftil$ increasing in $[0,\infty)$ because if not, we can proceed with studying the following modified problem and obtain the required results:
\begin{equation*}
    \begin{aligned}
        -\p u + \ps u -\la \frac{f(0)}{u^{\ba}} + \tilde{k}u &= \tilde{f}(u) +\tilde{k} u \text{ in } \Om\\
        u >0 \text{ in } \Om,\ u &=0 \text{ in } \bdry.
    \end{aligned}
\end{equation*}
Thus, in summary, we consider $\ftil$ as a continuous increasing function in $[0,\infty)$ with $\ftil (0) =0$. Here, we once again emphasize the fact that the above argument is necessary specifically for the singular case, as in the case when $\ba =0$, this issue is inherently addressed by the condition $(f2)$.

Now, we consider the following admissible set
\[ \mathcal{A} := \left\{ u\in \wop : u \in L^{\infty} (\Om),\ \phi_1\leq u \leq \psi_1 \text{ a.e. in } \Om \right\}, \]
where $\phi_1$ and $\psi_1$ are respectively sub- and supersolutions constructed in lemma \ref{first_pair}. We intend to define a solution map $T: \mathcal{A} \rar \mathcal{A} $ such that the fixed points of $T$ correspond to the weak solutions of our problem \eqref{P_lambda_tilde}.
\begin{definition}
    Define $T: \mathcal{A} \rar \mathcal{A} $ by $T(u) = z$ iff $z$ is a weak solution of
    \begin{equation}\tag{$S_{\la}$}\label{s_lambda}
        \begin{aligned}
        -\p z + \ps z &-\la \frac{f(0)}{z^{\ba}} = \tilde{f}(u) \text{ in } \Om\\
        z >0 \text{ in } \Om,\ z &=0 \text{ in } \bdry.
    \end{aligned}
    \end{equation}
\end{definition}
In the above definition, by saying $z$ is a solution of \eqref{s_lambda} we mean $z\in \wop$ with $ess\ inf_K z >0$ for every $K\Subset\Om$ and for every $\varphi\in\test$,
\begin{multline*}
    \hspace{-16pt}\int_{\Om} |\na z|^{p-2} \na z \cdot \na\varphi\ dx + \int_{\sr^N \times \sr^N} \hspace{-12pt} \frac{\left|z(x) - z(y)\right|^{p-2} \left(z(x)-z(y)\right) \left(\varphi (x) - \varphi (y)\right)}{|x-y|^{N+ps}}\ dx\ dy\\
    - \la f(0) \int_{\Om} \frac{\varphi}{z^{\ba}}\ dx = \int_{\Om} \tilde{f}(u) \varphi\ dx.
\end{multline*}
\begin{remark}
    If, additionally, we have $z \geq cd(x)$ in $\Om$ for some constant $c>0$, then the above identity holds for every $\varphi\in\wop,$ due to Hardy's inequality.
\end{remark}    

\begin{lemma}\label{singular_welldefined}
    The map $T: \mathcal{A} \rar \mathcal{A}$ is well-defined.
\end{lemma}

\begin{proof}
    We show that corresponding to each $u\in\mathcal{A}$, \eqref{s_lambda} has a unique solution in $\mathcal{A}$. Let $u\in\mathcal{A}$. To prove the existence of a solution, we consider the non-singular and singular cases separately.

    \textbf{Non-singular case:} The energy functional $J: \wop \rar \sr$ corresponding to \eqref{s_lambda} is given by
    \[J(v):= \frac{1}{p} \|v\|^p_{\mathcal{X}^{1,p}_0(\Om)} -\la \int_{\Om} f(u)v\ dx,\ v\in\wop. \]
    Since $f$ is continuous, $f(u)\in L^{\infty} (\Om) $. Thus, by standard minimization argument, $J$ admits a non-negative minimizer $z$ in $\wop$ and $z$ is a weak solution of \eqref{s_lambda}. Moreover, by Theorem 1.1 of \cite{antonini2025global}, we obtain $z\in C^{1,\alpha} \omc$, for some $\alpha\in (0,1)$ and since $f$ is non-decreasing with $f(0)>0$, the Strong Maximum Principle yields $z>0$ in $\Om$.

    \textbf{Singular case:} In this case, because of the lack of $C^1$ nature of the energy functional the direct minimization method cannot be applied. Hence, we establish the existence of a solution to \eqref{s_lambda} via approximation. For every $\varepsilon>0$, let us consider the following approximated problem:
    \begin{equation}\tag{$S_{\la}^{\varepsilon}$}\label{s_lambda_epsilon}
        \begin{aligned}
        -\p v + \ps v &-\la \frac{f(0)}{(v+ \varepsilon )^{\ba}} = \tilde{f}(u) \text{ in } \Om\\
        v >0 \text{ in } \Om,\ v &=0 \text{ in } \bdry.
        \end{aligned}
    \end{equation}
    The associated energy functional $J_{\varepsilon}: \wop \rar \sr$ is defined by
    \[ J_{\varepsilon}(v) = \frac{1}{p} \|v\|_{\mathcal{X}^{1,p}_0(\Om)}^p -\frac{\la f(0)}{(1-\ba)} \int_{\Om} \{(v+\varepsilon)^+ \}^{1-\ba}\ dx - \int_{\Om} \ftil (u)v\ dx,\ v\in\wop .\]
    Since $\ba<1$, $J_{\varepsilon}$ is coercive and weakly lower semicontinuous on $\wop$ and hence admits a minimizer $z_{\varepsilon} \in \wop$. Again, for every $v\in\wop, v\not\equiv 0$, we can choose $t>0$ small enough so that $J_{\varepsilon}(tv) <0$. Thus, $z_{\varepsilon}\not\equiv 0$. Moreover, the convexity of $J_{\varepsilon}$ guarantees that $z_{\varepsilon}$ is non-negative. Now, we make the following claim to show that $z_{\varepsilon}$ is a weak solution of \eqref{s_lambda_epsilon}.
    
    \noindent\textbf{Claim-1:} There exists $m_0 >0$ small enough such that $m_0 e \leq z_{\varepsilon}$ a.e. in $\Om$ for every $\varepsilon \in (0,1)$.
    
    We observe that whenever $m_0 >0$ is small enough, we have
    \begin{equation}\label{lower bound}
        -\p (m_0 e) + \ps (m_0 e) - \la \frac{f(0)}{\left(m_0 e+\varepsilon\right)^{\ba}} = m_0^{p-1} - \la \frac{f(0)}{\left(m_0 e+\varepsilon\right)^{\ba}} < 0 \leq \ftil (u)
    \end{equation}
    in $\Om$. Define $\tilde{z}_{\varepsilon}:= (m_0 e - z_{\varepsilon})^+$. If possible, let $supp\ \tilde{z}_{\varepsilon}$ have a positive measure. We then define $\eta: [0,1] \rar \sr$ by $\eta(t):= J_{\varepsilon} \left(z_{\varepsilon} + t\tilde{z}_{\varepsilon} \right)$. Since $\eta$ is convex and $z_{\varepsilon}$ is a minimizer of $J_{\varepsilon}$, we have $0\leq \eta' (0) \leq \eta' (1)$. Again, from \eqref{lower bound}, we obtain $\eta'(1) \leq \langle J_{\varepsilon}' (m_0 e), \tilde{z}_{\varepsilon} \rangle <0$, a contradiction. Therefore $supp\ \tilde{z}_{\varepsilon}$ is a set of measure zero or, equivalently, $m_0 e \leq z_{\varepsilon}$ a.e. in $\Om$ for every $\varepsilon \in (0,1)$. Thus Claim-1 is proved.

    Consequently, the energy functional $J_{\varepsilon} $ is G\^ateaux differentiable at $z_{\varepsilon}$ and hence $z_{\varepsilon}$ is a weak solution of \eqref{s_lambda_epsilon}. Moreover, since $\la \frac{f(0)}{\left(z_{\varepsilon} + \varepsilon \right)^{\ba}} + \tilde{f}(u) \in L^{\infty}(\Om) $, by Theorem 1.1 of \cite{antonini2025global} we obtain $z_{\varepsilon} \in C^{1,\alpha} \omc$ for some $\alpha\in (0,1)$.

    \noindent\textbf{Claim-2:} The sequence $\{z_{\varepsilon}\}$ is monotonically increasing, i.e, if $0<\varepsilon<\varepsilon'$ then $z_{\varepsilon'}\leq z_{\varepsilon}$, and $\{z_{\varepsilon}\}$ is bounded above.

    The monotonicity of the sequence $\{z_{\varepsilon}\}$ follows directly from the weak comparison principle by taking $\left(z_{\varepsilon'} - z_{\varepsilon}\right)^+$ as the test function. To prove $\{z_{\varepsilon}\}$ is bounded above, consider $\overline{z}= M_0 \psi_0$, where $\psi_0$ solves equation \eqref{psi_0} and $M_0 >0$ is chosen large enough so that $M_0^{p-1} \left( 2 - \frac{\la f(0)}{M_0^{\ba}} \right) \frac{1}{\|\psi_0\|_{\infty}^{\ba}} \geq \ftil \left(\|u\|_{\infty}\right)$. Then, we have
    \begin{equation*}
        -\p \overline{z} + \ps \overline{z} -\la \frac{f(0)}{\left(\overline{z} + \varepsilon\right)^{\ba}} = \frac{2M_0^{p-1}}{\psi_0^{\ba}} - \la \frac{f(0)}{\left(M_0\psi_0 + \varepsilon\right)^{\ba}}\geq \ftil (u) \text{ in } \Om.
    \end{equation*}
    Therefore, by weak comparison principle, we obtain $z_{\varepsilon} \leq \overline{z}$ in $\Om$ for every $\varepsilon>0$ and this proves Claim-2.

    We now consider the pointwise limit $z := \lim_{\varepsilon\rar 0+} z_{\varepsilon}$. From the weak formulation of \eqref{s_lambda_epsilon}, we have,
    \begin{align*}
        \|z_{\varepsilon}\|^p_{\mathcal{X}^{1,p}_0(\Om)} &\leq \la f(0) \int_{\Om} z_{\varepsilon}^{1-\ba} dx + \int_{\Om} \ftil (u) z_{\varepsilon}\ dx\\
        &\leq \la f(0) \int_{\Om} \overline{z}^{1-\ba} dx + \int_{\Om} \ftil (u) \overline{z}\ dx < \infty
    \end{align*}
    as $\overline{z} = M_0 \psi_0 \in L^{\infty}(\Om)$. Thus $\{z_{\varepsilon} \}$ is bounded in $\wop$ and upto a subsequence $z_{\varepsilon} \rightharpoonup z$ in $\wop$. Thanks to Hardy's inequality and the Dominated Convergence Theorem, we can now pass through the limit in the weak formulation of \eqref{s_lambda_epsilon} and we obtain
    \begin{equation*}
        -\p z + \ps z -\la \frac{f(0)}{z^{\ba}} = \tilde{f}(u) \text{ in } \Om,
    \end{equation*}
    that is, $z$ is a weak solution of \eqref{s_lambda}. Additionally we have $m_0 e \leq z \leq M_0 \psi_0$ a.e. in $\Om$. Hence $z \in L^{\infty}(\Om)$ and $ess\ inf_K z >0$ for every $K\Subset\Om$.

    Therefore, in both (non-singular and singular) cases, for every $u\in\mathcal{A}$, there exists $z\in L^{\infty}(\Om)$, $z>0$ a.e. in $\Om$ such that $z$ solves \eqref{s_lambda}. Also, the uniqueness of $z$ follows from a direct application of the weak comparison principle. Thus, to complete the proof, it only remains to verify that $z\in\mathcal{A}$.

    Since $\phi_1$ is a subsolution of \eqref{s_lambda}, $u\in\mathcal{A}$ and $\ftil$ is monotone increasing, we have
    \begin{equation*}
        -\p \phi_1 + \ps \phi_1 -\la \frac{f(0)}{\phi_1^{\ba}} \leq \tilde{f}(\phi_1) \leq \tilde{f}(u) = -\p z + \ps z -\la \frac{f(0)}{z^{\ba}} \text{ in } \Om.
    \end{equation*}
    Taking $(\phi_1 - z)^+$ as the test function,  we obtain $\phi_1 \leq z$ a.e. in $\Om$. In a similar way, we get $z \leq \psi_1$ a.e. in $\Om$. Therefore, $z\in \mathcal{A}$ and the map $T:\mathcal{A}\rar\mathcal{A}$ is well-defined.
\end{proof}

Now, we use the monotone iteration method to establish the existence of weak solutions to \eqref{P_lambda} or, equivalently, \eqref{P_lambda_tilde}.

\noindent\textbf{Proof of Theorem \ref{theorem_1}:} (a) We show that for every $\la>0$, \eqref{P_lambda_tilde} admits a positive solution. Fix some $\la>0$. Define $w_0 := \phi_1$ and $w_1 := T(\phi_1)$. Then $w_1 \in \mathcal{A}$ and $\phi_1 = w_0 \leq w_1 \leq \psi_1$. Next, define $w_2 := T(w_1)$. Then $w_2 \in \mathcal{A}$ and
\begin{equation*}
    -\p w_1 + \ps w_1 -\la \frac{f(0)}{w_1^{\ba}} = \ftil (w_0) \leq \ftil (w_1) = -\p w_2 + \ps w_2 -\la \frac{f(0)}{w_2^{\ba}}
\end{equation*}
so that by weak comparison principle, $\phi_1 = w_0 \leq w_1 \leq w_2 \leq \psi_1$ in $\Om$.

Proceeding in this way, we construct a sequence $\{w_n\}$ in $\mathcal{A}$ by setting $w_0 = \phi_1$ and $w_n = T(w_{n-1})$, for all $n\geq 1$. Then we have
\begin{equation}\label{monotone_iteration}
    \begin{aligned}
        -\p w_n + \ps w_n -\la \frac{f(0)}{w_n^{\ba}} &= \ftil (w_{n-1}) \text{ in } \Om\\
        w_n >0 \text{ in } \Om,\ w_n &=0 \text{ in } \bdry
    \end{aligned}
\end{equation}
and by the weak comparison principle,
\[ \phi_1 = w_0 \leq w_1 \leq w_2 \leq \dots \leq w_{n-1} \leq w_n \leq \dots \leq \psi_1 \text{ in } \Om.\]
Hence the pointwise limit of $\{w_n\}$ exists and let $z:= \lim_{n\rar\infty} w_n$.

Since $w_n \leq \psi_1$, $\forall\ n\in\sn$ and $\psi_1\in L^{\infty}(\Om)$, $\{\ftil (w_n)\}_{n\geq 0}$ is bounded in $\sr$. Consequently, $\{w_n\}$ is bounded in $\wop$. So upto a subsequence, $w_n \rightharpoonup z$ in $\wop$. Now fix any test function $\varphi\in \test$ and choose an open subset $\Om_0$ of $\Om$ such that $supp\ \varphi \Subset \Om_0 \Subset \Om$. Then by the interior regularity result (Theorem 5) of \cite{de2024gradient}, $\{ w_n\}$ is bounded in $C^{1,\alpha} (\overline{\Om}_0)$ and due to the compact embedding, $w_n \rar z$ in $C^1 (\overline{\Om}_0)$, upto a subsequence. Taking limit in \eqref{monotone_iteration} as $n\rar\infty$, we get
\begin{equation*}
    -\p z + \ps z -\la \frac{f(0)}{z^{\ba}} = \ftil (z) \text{ in } \Om.
\end{equation*}
Thus $z$ is a weak solution of \eqref{P_lambda}. Moreover, as $z \leq \psi_1$ a.e. in $\Om$, $z\in L^{\infty} (\Om)$. Therefore, for every $\la>0$, \eqref{P_lambda} admits a positive solution $z$ in $L^{\infty} (\Om)$.

(b) Since when $\la\in (\la_* , \la^*)$, we have two pairs $(\phi_1, \psi_1)$ and $(\phi_2, \psi_2)$ of sub- and supersolutions satisfying $\phi_1\leq \phi_2 \leq \psi_1$, $\phi_1\leq \psi_2 \leq \psi_1$ and $\phi_2 \not\leq\psi_2$ in $\Om$, we consider the following two admissible sets:
\begin{equation*}
    \begin{aligned}
        \mathcal{A}_1 := \left\{ u\in \wop : u \in L^{\infty} (\Om),\ \phi_1\leq u \leq \psi_2 \text{ a.e. in } \Om \right\},\\
        \mathcal{A}_2 := \left\{ u\in \wop : u \in L^{\infty} (\Om),\ \phi_2\leq u \leq \psi_1 \text{ a.e. in } \Om \right\}.
    \end{aligned}
\end{equation*}
Repeating the arguments of part (a) in each of $\mathcal{A}_1$ and $\mathcal{A}_2$, we obtain two solutions $z_1 \in \mathcal{A}_1$ and $z_2\in \mathcal{A}_2$. Also because of the condition $\phi_2 \not\leq\psi_2$ in $\Om$, the sets $\mathcal{A}_1$ and $\mathcal{A}_2$ are disjoint and so $z_1$ and $z_2$ are also distinct. Therefore, for $\la\in (\la_* , \la^*)$, \eqref{P_lambda} admits two positive solutions in $L^{\infty} (\Om)$.\hfill\qed

\begin{remark}\label{regularity}
    In the above proof, the solutions $z_1, z_2$ are obtained in $L^{\infty} (\Om)$. However, in light of the recent global regularity results of \cite{antonini2025global} and \cite{dhanya2024interiorboundaryregularitymixed}, \cite{bal2025regularityresultsclassmixed}, these solutions actually belong to $C^{1,\alpha} \omc$, for some $\alpha\in (0,1)$.
\end{remark}


\section{A Third Solution in The Non-singular Case}\label{section-third solution}

In this section, we further extend the previous result and establish the existence of a third solution for \eqref{P_lambda} in the specific case of $\ba =0$, by employing a fixed point theorem due to Amann.
\begin{proposition}{(see Amann \cite{amann1976fixed})}\label{amann}
    Let $X$ be a retract of some Banach space and $f: X \rar X$ be a completely continuous map. Suppose $X_1$ and $X_2$ are disjoint retracts of $X$ and let $U_k$, $k=1,2$, be open subsets of $X$ such that $U_k \subset X_k$, $k=1,2$. Moreover, suppose that $f(X_k)\subset X_k$ and that $f$ has no fixed points on $X_k \setminus U_k$, $k=1,2$. Then $f$ has at least three distinct fixed points $x, x_1, x_2$ with $x_k \in X_k$, $k=1,2$ and $x\in X \setminus (X_1 \cup X_2)$.
\end{proposition}

We recall the solution map $T$ defined in Section \ref{section-multiplicity}. The primary focus is now to establish certain key properties of the map $T$, which will act as sufficient conditions for verifying the hypotheses of Proposition \ref{amann}. A crucial component of this process is the Hopf-type Strong Comparison Principle, which we introduce first before proceeding with the proof of Theorem \ref{theorem_2}.

\begin{proposition}\label{nonsingular_SCP}
    Let $g\in C_0 (\sr)\cap BV_{loc}(\sr) $, $u,v \in \wop\cap C_0 \omc$, $u \not\equiv v$ and there exists $K>0$ such that
    \begin{equation*}
        \begin{aligned}
            -\p u + \ps u +g(u) &\leq -\p v + \ps v +g(v) \leq K \text{ in } \Om\\
            0< u &\leq v \text{ in } \Om.
        \end{aligned}
    \end{equation*}
    Then $u<v$ in $\Om$ and $v-u \in C_e^+ \omc$, where $e$ is defined by equation \eqref{e}.
\end{proposition}

The proof of the above Proposition follows almost verbatim from that of Theorem 2.7 of \cite{iannizzotto2023logistic}, where an analogous result is established for the fractional $p$-Laplacian. Hence, we omit the details here.

With all the required tools now in place, we move on to the proof of Theorem \ref{theorem_2}.

\noindent\textbf{Proof of Theorem \ref{theorem_2}:} We show that for every $\la\in (\la_* , \la^*)$, \eqref{P_lambda} admits three distinct positive solutions in $C^{1,\alpha} \omc$. So we fix some $\la\in (\la_* , \la^*)$ and in view of remark \ref{regularity} we define
\begin{equation*}
    \begin{aligned}
        X:= \left\{u\in C_0 \omc : \phi_1 \leq u \leq \psi_1 \text{ in } \Om \right\},\\
        X_1:= \left\{u\in C_0 \omc : \phi_1 \leq u \leq \psi_2 \text{ in } \Om \right\},\\
        X_2:= \left\{u\in C_0 \omc : \phi_2 \leq u \leq \psi_1 \text{ in } \Om \right\}.
    \end{aligned}
\end{equation*}
Then we have the following observations:
\begin{enumerate}
    \item Each of $X$, $X_1$ and $X_2$ is a non-empty, closed and convex subset of the Banach space $C_e \omc$, the function $e$ being defined in \eqref{e}.

    \item Since $\phi_2 \not\leq \psi_2$ in $\Om$, the sets $X_1$ and $X_2$ are disjoint.

    \item By the weak comparison principle, we have $T(X) \subset X$, $T(X_1) \subset X_1$ and $T(X_2) \subset X_2$.
\end{enumerate}
Now, following the approach used in the proof of part (a), Theorem \ref{theorem_1}, we employ the monotone iteration method in $X_1$ and $X_2$ starting from $\psi_2$ and $\phi_2$ respectively. This process yields two positive solutions $z_1 \in X_1$ and $z_2 \in X_2$ for \eqref{P_lambda}. Moreover, these specific choices of the initial terms ensure that $z_1$ is the maximal positive solution in $X_1$ and $z_2$ is the minimal positive solution in $X_2$. Now, to obtain a third solution, we invoke Proposition \ref{amann} and proceed to verify its hypotheses.

\noindent\textbf{Claim-1:} The map $T: X \rar X$ is completely continuous.

Consider a sequence $u_n \rar u$ in $X$ as $n\rar\infty$. Let $w=T(u)$ and $w_n = T(u_n),\ n\in\sn$, that is, $w_n$ satisfies
\begin{equation}\label{nonsingular_continuity}
    \begin{aligned}
        -\p w_n + \ps w_n = \la f(u_n) \text{ in } \Om\\
        w_n >0 \text{ in } \Om,\ w_n = 0 \text{ in } \bdry.
    \end{aligned}
\end{equation}
Since $f$ is continuous, there exists a constant $C>0$ such that $\la f(u_n) \leq C$, $\forall\ n\in\sn$. Consequently, by the global regularity result (Theorem 1.1) of \cite{antonini2025global}, the sequence $\{ w_n\}$ is bounded in $C^{1,\alpha} \omc$ for some $\alpha \in (0,1)$. By the compact embedding, upto a subsequence, $\{ w_n\}$ converges in $C^{1,\alpha'} \omc$ for some $0<\alpha'<\alpha$. Therefore, passing to the limit in \eqref{nonsingular_continuity} as $n\rar\infty$ and using the uniqueness of solutions, we obtain $w_n \rar w$ as $n\rar\infty$. In fact, every subsequence of $\{w_n\}$ has a further subsequence converging to $w$ in $C^{1,\alpha'} \omc$ as $n\rar\infty$. Consequently, the map $T: X \rar X$ is continuous.

To prove the compactness, let us take a bounded sequence $\{u_n\}$ in $X$ and let $T(u_n) = w_n,\ n\in\sn$. Then a similar argument yields that $T: X \rar C^{1,\alpha'}_0 \omc$ is compact. Since $C^{1,\alpha'}_0 \omc \hookrightarrow C_0 \omc$ and $T(X) \subset X$, it follows that the map $T: X \rar X$ is completely continuous.

Next, we make the following claim and proving this will confirm that all hypotheses of Proposition \ref{amann} are satisfied.

\noindent\textbf{Claim-2:} There exist open subsets $U_1, U_2$ of $X$ such that $U_1 \subset X_1$, $U_2 \subset X_2$ and $T$ has no fixed points in $(X_1 \setminus U_1) \cup (X_2 \setminus U_2)$.

Consider the set $X_1$. From equation \eqref{sub_soln_1} and part (b) of Theorem \ref{theorem_1} recall that the functions $\phi_1$ and $z_1$ satisfy the following equations:
\begin{equation*}
    \begin{aligned}
        -\p \phi_1 + \ps \phi_1 &< \la f(\phi_1) \text{ in } \Om,\\
        -\p z_1 + \ps z_1 &= \la f(z_1) \text{ in } \Om.
    \end{aligned}
\end{equation*}
Therefore, due to the monotonicity of $f$, we have
\[-\p \phi_1 + \ps \phi_1 \leq -\p z_1 + \ps z_1 \leq K \text{ in } \Om\]
where $K= \la f(\|\psi_2 \|_{\infty})$. Then, by Proposition \ref{nonsingular_SCP}, we obtain $z_1 - \phi_1 \in C_e^+ \omc$. Similarly, we get $\psi_2 - z_1 \in C_e^+ \omc$. Now, since $z_1$ is the maximal positive solution of \eqref{P_lambda} in $X_1$, we define $U_1$ to be the largest open set in $X_1$ containing $z_1$ such that $X_1 \setminus U_1$ contains no fixed point of the map $T$. Similarly, we define $U_2$ to be the largest open set in $X_2$ containing $z_2$ such that $X_2 \setminus U_2$ contains no fixed point of the map $T$. This proves our claim.

Thus, by Proposition \ref{amann}, $T$ has another fixed point $z_3 \in X \setminus (X_1 \cup X_2) $. Consequently, the problem \eqref{P_lambda} has three distinct positive solutions $z_1 \in X_1$, $z_2 \in X_2$ and $z_3 \in X \setminus (X_1 \cup X_2) $ whenever $\la\in (\la_* , \la^*)$.\hfill\qed


\section{A Third Solution in The Singular Case for The Linear Operator}\label{section-third solution- linear singular case}

In this case, \eqref{P_lambda} simplifies to the following form:
\begin{equation}\label{P_lambda_linear}
    \begin{aligned}
        -\Delta u + (-\De)^s u = \la \frac{f(u)}{u^{\ba}} \text{ in } \Om\\
        u >0 \text{ in } \Om,\ u =0 \text{ in } \bdry
    \end{aligned}
\end{equation}
where $0<\ba<1$. While the overall idea remains similar to that employed in the previous section, the inclusion of the singular term introduces additional challenges requiring more refined analysis. From the proof of Theorem \ref{theorem_2}, it is evident that the main difficulty lies in proving the strongly increasing nature of the map $T$, that is, if $u_1 \leq u_2$ with $u_1 \not\equiv u_2$ in $\Om$, then $T(u_2)- T(u_1) \in\ C_e^+ (\overline{\Om})$. This stronger condition, beyond mere monotonicity, is essential for verifying the hypotheses of Amann's fixed point theorem, specifically Claim-2 in the proof of Theorem \ref{theorem_2}.

In \cite{giacomoni2019existence}, the authors successfully handled the singular case ($0<\beta<1$) for the fractional Laplacian by analyzing an infinite semipositone problem. However, even in this linear case, substantial analysis was required to obtain a positive solution, behaving like the function $d^s(x)$, for the corresponding infinite semipositone problem. Next, weak comparison principle is applied to show that this solution acts as a lower bound for $T(u_2) - T(u_1)$ in $\Omega$ and consequently the map $T$ becomes strongly increasing.

Here, we provide a different approach to derive a Strong Comparison Principle for singular problems involving the mixed local-nonlocal Laplacian (Theorem \ref{Hopf lemma_linear case}). The proof proceeds by first obtaining a strict comparison between the functions themselves, and subsequently between their normal derivatives on the boundary of $\Om$. We note the contribution of \cite{dhanya2024interiorboundaryregularitymixed} where a Strong Comparison Principle (Theorem 2.8) is established for mixed operators, under the restriction of $s\in \left(0, \frac{1}{2}\right)$ in the linear case. Here, thanks to the linearity of the operator, we employ a weak Harnack-type inequality that enables us to derive the desired result for the full range $s\in (0,1)$. The final step consists of proving the Hopf-type Strong Comparison Principle, where the key idea is to construct a function $v$ that mimics the behavior of $d(x)$ near $\pa\Om$ and satisfies $T(u_2) - T(u_1) \geq Cv(x)$ in $\Om$ for some $C>0$.

\begin{lemma}\label{strict comparison}
    Let $u_1, u_2 \in C^{1,\alpha}(\overline{\Om})$, for some $\alpha\in(0,1)$, $u_1, u_2 > 0$ in $\Om$ and they satisfy \eqref{scp_equns} with $f_1, f_2$ as in Theorem \ref{Hopf lemma_linear case}. Then $u_1<u_2$ in $\Om$.
\end{lemma}

\begin{proof}
    Define $w:= u_2 - u_1$. Then $w \geq 0$ in $\sr^N$ and by the mean value theorem, there exists a function $\xi \in [u_1, u_2]$ such that $w$ satisfies the following:
    \begin{equation*}
        -\Delta w + (-\De)^s w + \frac{\beta}{\xi (x) ^{\ba +1}} w = f_2 - f_1 \geq 0 \text{ in } \Om.
    \end{equation*}
    Since $u_1, u_2$ satisfy \eqref{scp_equns}, it follows that $u_1, u_2 \sim d(x)$ in $\Om$. Hence there exists $k>0$ such that
    \begin{equation}\label{w_equn}
        -\Delta w + (-\De)^s w + \frac{k}{d(x)^{\ba +1}} w \geq 0 \text{ in } \Om.
    \end{equation}
    Suppose, towards a contradiction, that there exists a point $x_0 \in \Om$ with $w(x_0) =0$. Fix $0<r<1$ such that $B_{2r} (x_0) \Subset \Om$. Then, for some constant $k_1 >0$, depending on $x_0$ and $r$, we have
    \begin{equation*}
        -\Delta w + (-\De)^s w + k_1 w \geq 0 \text{ in } B_{2r} (x_0)
    \end{equation*}
    in the sense of definition 3.1 of \cite{biagi2024multiplicity}. Since $w \geq 0 \text{ in } \sr^N$, the weak Harnack-type inequality (Proposition 3.3, \cite{biagi2024multiplicity}) yields positive constants $k_2$, $k_3$ such that
    \begin{equation*}
        \left( \dashint_{B_r (x_0)} w^{k_2}\ dx \right)^{1/k_2} \leq k_3\ ess\ inf_{B_r (x_0)} w.
    \end{equation*}
    Consequently, $w=0$ in $B_r (x_0)$.

    To reach a contradiction, we test the equations in \eqref{scp_equns} with a non-negative $\varphi \in \test$, $\varphi \not\equiv 0$, with $supp\ \varphi \Subset B_r (x_0)$ and subtract. This yields
    \begin{multline*}
        \int_{\Om} \na w \cdot \na \varphi\ dx + \int_{\sr^N \times \sr^N} \frac{\left(w(x) - w(y) \right) \left(\varphi(x) - \varphi(y) \right)}{|x-y|^{N+2s}}\ dx\ dy\\
        - \int_{\Om} \left( \frac{1}{u_2^{\beta}} - \frac{1}{u_1^{\beta}} \right) \varphi\ dx = \int_{\Om} \left( f_2-f_1 \right) \varphi\ dx.
    \end{multline*}
    Due to the choice of $\varphi$, the first and third terms on the left-hand side vanish. Moreover, since neither $w$ nor $\varphi$ are identically zero in $\Om$ and $f_2 \geq f_1$, we arrive at
    \begin{equation*}
        0\leq \int_{\Om} \left( f_2-f_1 \right) \varphi\ dx = \int_{\sr^N \times \sr^N} \frac{\left(w(x) - w(y) \right) \left(\varphi(x) - \varphi(y) \right)}{|x-y|^{N+2s}}\ dx\ dy < 0,
    \end{equation*}
    which is not possible. Therefore, $w>0$ in $\Om$, or equivalently, $u_1 < u_2$ in $\Om$.
\end{proof}

Next, we present an important observation concerning the $C^{1,\alpha}$-extension of a suitable function to the entire $\sr^N$. This extension serves as a key component in proving Theorem \ref{Hopf lemma_linear case} (Hopf-type SCP).

We use the following notations: for any $x\in \sr^N$ we write $x= (x_1, x_2, \dots,$ $ x_N) = (x', x_N)$ and define $\sr^N_+ = \left\{x\in \sr^N : x_N >0\right\}$ and $\sr^N_- = \left\{ x\in \sr^N : x_N <0\right\}$. For any $r>0$ and $x\in \sr^N$, $B_r (x)$ denotes the ball of radius $r$ centered at $x$ with $B_r^+ (x) = B_r (x) \cap \sr^N_+ $ and $B_r^- (x) = B_r (x) \cap \sr^N_-$.

\begin{lemma}\label{extension of e}
    Let $e_0 \in C^{1,\alpha} \omc$, $\alpha \in (0,1)$, be the unique solution of
    \begin{equation}\label{e_0 equation}
        \begin{aligned}
            -\Delta e_0 &=1 \text{ in } \Om\\
            e_0 &= 0 \text{ on } \pa\Om.
        \end{aligned}
    \end{equation}
    Then there exists an extension $\tilde{e}_0$ of $e_0$ such that $\tilde{e}_0 \in C^{1,\alpha} (\sr^N)$, $\tilde{e}_0 = e_0$ in $\overline{\Om}$ and $\tilde{e}_0 \leq 0$ in $\sr^N \setminus \Om$.
\end{lemma}

\begin{proof}
    Let $x_0$ be a point on $\pa\Om$. At first, we extend the function $e_0$ in a neighborhood of the point $x_0$ by using the straightening of boundary technique. Since $\pa\Om$ is smooth, for chosen $x_0 \in \pa\Om$, there exist $r_0 >0$ and a $C^1$-diffeomorphism $\Psi_{x_0} : B_{r_0} (x_0) \rar \Psi_{x_0} \left(B_{r_0} (x_0) \right)$ such that
    \begin{equation*}
        \begin{aligned}
            \Psi_{x_0} \left(\Om \cap B_{r_0} (x_0) \right) &\subset \sr_+^N\\
            \text{and } \Psi_{x_0} \left(\pa \Om \cap B_{r_0} (x_0) \right) &\subset \left\{x\in \sr^N : x_N =0 \right\}.
        \end{aligned}
    \end{equation*}
    Without loss of generality, we assume $\Psi_{x_0} (x_0) =0$ and $B_1^+ (0) \subset$ $\Psi_{x_0} \left(B_{r_0} (x_0) \right)$.

    Define $q_{x_0} (y):= \left(e \circ \Psi_{x_0}^{-1}\right) \left(y\right)$, for $y \in B_1^+ (0)$. We extend $q_{x_0} (y)$ to a function $\tilde{q}_{x_0} (y)$ defined on the entire $B_1 (0)$ by the method of reflection. Let $y=(y_1, y_2, \dots, y_N)= (y', y_N) \in B_1 (0)$. For $y_N<0$ we define
    \[\tilde{q}_{x_0} (y', y_N) = - \tilde{q}_{x_0} (y', -y_N), \text{ for } y\in B_1^- (0).\]
    With this definition, it is now straightforward to check that $\tilde{q}_{x_0}$$\in C^{1,\alpha} (B_1 (0))$. Therefore, $\tilde{e}_{x_0} (x)= \left(\tilde{q}_{x_0} \circ \Psi_{x_0}\right) \left(x\right)$ is a $C^{1,\alpha}$-extension of $e_0$ to $\Psi_{x_0}^{-1} \left(B_1 (0)\right)$. Moreover, as $e_0 > 0$ in $\Om$, it is evident from the definition that $\tilde{e}_{x_0} \leq 0$ in $\Psi_{x_0}^{-1} \left(B_1^- (0)\right)$.

    Since $\pa\Om$ is compact, let $\left\{ U_i = \Psi_{x_i}^{-1} \left(B_1 (0)\right) : i=1,2, \dots, k \right\}$ be a finite subcover of $\pa\Om$, where $x_i \in \pa\Om$, $i=1,2, \dots, k$. Let $U_0 = \Om$. Then $\left\{ U_i \right\}_{i=0}^k$ is a cover of $\overline{\Om}$. Consider a smooth partition of unity $\left\{\xi_i\right\}_{i=0}^k$ subordinate to this cover $\left\{U_i\right\}_{i=0}^k$, that is, $\xi_i \in C_c^{\infty}(U_i)$ satisfying $0\leq \xi_i \leq 1$, $i= 0, 1, 2, \dots, k$ and $\sum_{i=0}^k \xi_i = 1$ on $\overline{\Om}$. Let $\tilde{e}_{x_i}$ be the $C^{1,\alpha}$-extension of $e_0$ to $U_i$, $i=1,2, \dots,k$ as constructed above. Define
    \[\tilde{e}_0 := \xi_0 e_0 + \sum_{i=1}^k \xi_i \tilde{e}_{x_i}. \]
    Then, clearly $\tilde{e}_0\in C^{1,\alpha} (\sr^N)$ with $\tilde{e}_0 = e_0$ in $\overline{\Om}$, there exists an open set $\Om_1$ such that $\Om \Subset \Om_1$, $\tilde{e}_0 \leq 0$ in $\Om_1 \setminus \Om$ and $\tilde{e}_0 = 0$ in $\sr^N \setminus \Om_1$. This completes the proof.    
\end{proof}

Now, we turn to the proof of the Hopf-type Strong Comparison Principle. We first treat the case $s\in \left(0, \frac{1}{2}\right)$ and then extend the reasoning to the case $s\in \left[\frac{1}{2}, 1\right)$. Though the essential structure of the proof is common to both cases, the latter range introduces subtleties and requires a more delicate estimate for the nonlocal term.

\begin{lemma}\label{hopf_first}
    Let $s \in \left(0, \frac{1}{2}\right)$, $u_1, u_2 \in C^{1,\alpha}(\overline{\Om})$, for some $\alpha\in(0,1)$, $u_1, u_2 > 0$ in $\Om$ and satisfy \eqref{scp_equns} with $f_1, f_2$ as in Theorem \ref{Hopf lemma_linear case}. Then $\frac{\pa u_2}{\pa\nu} < \frac{\pa u_1}{\pa\nu} < 0$ on $\pa\Om$.
\end{lemma}

\begin{proof}
    From the weak comparison principle and the Hopf lemma for mixed operators (Theorem 1.2, \cite{antonini2025global}), it follows that $\frac{\pa u_2}{\pa\nu} \leq \frac{\pa u_1}{\pa\nu} < 0$ on $\pa\Om$. To refine this result and obtain the strict inequality $\frac{\pa u_2}{\pa\nu} < \frac{\pa u_1}{\pa\nu}$ on $\pa\Om$, we define $w:= u_2 - u_1$. By lemma \ref{strict comparison}, $w > 0$ in $\Om$ and satisfies \eqref{w_equn}.

    Define $v(x):= \et (x) + e_0 (x)^{\gamma}$, where $\et$ is the extension given by lemma \ref{extension of e}, $e_0$ (by abuse of notation) denotes the function solving \eqref{e_0 equation} with zero extension in $\sr^N \setminus \Om$ and the exponent $\gamma\in (1,2)$ is to be chosen later. Then
    \begin{equation}\label{direct computation}
        -\Delta v + \frac{k}{d(x)^{\ba +1}} v = 1 + \gamma e_0^{\gamma -1} -\gamma (\gamma -1) e_0^{\gamma -2} |\na e_0|^2 + \frac{k}{d(x)^{\ba +1}} \left(\et + e_0^{\gamma}\right) \text{ in } \Om.
    \end{equation}
    Since $\et (x) = e_0 (x) \sim d(x)$ in $\Om$ and $\inf_{\pa\Om} \left| \frac{\pa e_0}{\pa\nu} \right| >0$, choosing $\eta>0$ small enough and $\gamma \in (1, 2- \ba)$, we obtain
    \begin{equation}\label{est_k1}
        -\Delta v + \frac{k}{d(x)^{\ba +1}} v \leq - \frac{k_1}{d(x)^{2-\gamma}} \text{ in } \Om_{\eta},
    \end{equation}
    where $\Om_{\eta} = \left\{ x\in \Om : d(x) < \eta \right\}$ and $k_1 >0$ is a constant. Moreover, since $v \in C^{1,\alpha} (\sr^N)$ and $s\in \left(0, \frac{1}{2}\right)$, by Proposition 2.6 of \cite{silvestre2007regularity}, there exists a constant $k_2 >0$ such that $\left|(-\De)^s v\right| \leq k_2$ in $\overline{\Om}_{\eta}$. Therefore, choosing $\eta$ smaller if required, we get
    \begin{equation}\label{v_equn}
        -\Delta v + (-\De)^s v + \frac{k}{d(x)^{\ba +1}} v \leq 0 \text{ in } \Om_{\eta}.
    \end{equation}
    Since by lemma \ref{strict comparison}, $w >0$ in $\Om$, we choose $\epsilon>0$ small enough so that $\epsilon v \leq w$ in $\Om \setminus \Om_{\eta}$. Then, from \eqref{w_equn} and \eqref{v_equn} we obtain
    \begin{equation}
        \begin{aligned}
            -\Delta w + (-\De)^s w + \frac{k}{d(x)^{\ba +1}} w &\geq 0 \text{ in } \Om_{\eta}\\
            -\Delta (\epsilon v) + (-\De)^s (\epsilon v) + \frac{k}{d(x)^{\ba +1}} (\epsilon v) &\leq 0 \text{ in } \Om_{\eta}\\
            w \geq \epsilon v \text{ in } \sr^N \setminus \Om_{\eta}.
        \end{aligned}
    \end{equation}
    Thus, by weak comparison principle, we have $w \geq \epsilon v = \epsilon \left(\et + e_0^{\gamma}\right)$ in $\Om_{\eta}$. Since $\et = e_0$ in $\Om$, we thus have $\frac{\pa w}{\pa\nu} \leq \frac{\pa \left(\epsilon e_0\right)}{\pa\nu} < 0$ or, equivalently $\frac{\pa u_2}{\pa\nu} < \frac{\pa u_1}{\pa\nu}$ on $\pa\Om$. This completes the proof.
\end{proof}

\begin{lemma}\label{hopf_second}
    Let $s \in \left[\frac{1}{2}, 1\right)$, $u_1, u_2 \in C^{1,\alpha}(\overline{\Om})$, for some $\alpha\in(0,1)$, $u_1, u_2 > 0$ in $\Om$ and satisfy \eqref{scp_equns} with $f_1, f_2$ as in Theorem \ref{Hopf lemma_linear case}. Then $\frac{\pa u_2}{\pa\nu} < \frac{\pa u_1}{\pa\nu} < 0$ on $\pa\Om$.
\end{lemma}

\begin{proof}
    We define the functions $w$ and $v$ as in lemma \ref{hopf_first}. Recall that they satisfy, respectively, equations \eqref{w_equn} and \eqref{est_k1} in $\Om_{\eta}$, for $\eta>0$ small and $\gamma \in \left(1, 2- \ba\right)$.

    \noindent\textbf{Claim:} For sufficiently small $\eta>0$, there exists $k_2 >0$ such that for all $x\in \Om_{\eta}$,
    \[\left|(-\De)^s v (x)\right| \leq \frac{k_2}{d(x)^{2s-1}}, \text{ for } s\in \left[\frac{1}{2},1\right).\]

    To establish this, we first estimate the Hessian of $v$. For any $i,j= 1,2, \dots,$ $N$, we have
    \begin{equation*}
        \frac{\pa^2}{\pa x_i \pa x_j} \left(e_0^{\gamma}\right) = \gamma \left( \gamma -1 \right) e_0^{\gamma-2} \frac{\pa e_0}{\pa x_i} \frac{\pa e_0}{\pa x_j} + \gamma e_0^{\gamma-1} \frac{\pa^2 e_0}{\pa x_i \pa x_j} \text{ in } \Om_{\eta}.
    \end{equation*}
    Since $1< \gamma <2$ and $e_0$ is the unique solution of equation \eqref{e_0 equation}, it follows that
    \begin{equation}\label{Hessian_estimate}
        \left|D^2v (x)\right| \leq \frac{k_3}{d(x)^{2-\gamma}} \leq \frac{k_3}{d(x)} \text{ in } \Om_{\eta}, \text{ for } \eta>0 \text{ small enough}.
    \end{equation}
    Next, for $x\in \Om_{\eta}$, we decompose the nonlocal term as
    \begin{align}\label{est_nonlocal}
        (-\De)^s v (x) &= -\frac{1}{2} \int_{|z| \leq \frac{1}{2}d(x)} \hspace{-5pt} \frac{v(x+z) + v(x-z) - 2v(x)}{|z|^{N+2s}}\ dz + \int_{|z| > \frac{1}{2}d(x)} \hspace{-5pt} \frac{v(x) - v(x-z)}{|z|^{N+2s}}\ dz\nonumber\\
        &=: I_1 + I_2 \text{ (say)}.
    \end{align}
    For $I_1$, Taylor's expansion yields $\theta_1, \theta_2 \in (0,1)$ such that
    \begin{align*}
        |I_1| &= \left|\frac{1}{4} \int_{|z| \leq \frac{1}{2}d(x)} \frac{z^T D^2v (x+ \theta_1 z) z + z^T D^2v (x- \theta_2 z) z}{|z|^{N+2s}}\ dz \right|\\
        &\leq \frac{k_3}{4} \int_{|z| \leq \frac{1}{2}d(x)} \left( \frac{1}{d(x+ \theta_1 z)} + \frac{1}{d(x- \theta_2 z)} \right) |z|^{-N-2s+2}\ dz, \text{ by \eqref{Hessian_estimate}}.
    \end{align*}
    Note that for every $x\in \Om_{\eta}$ and $|z| \leq \frac{1}{2}d(x)$, the points $x\pm \theta_i z \in \Om$ and $d(x\pm \theta_i z) \geq \frac{1}{2}d(x)$, for $i=1,2$. As a result, we obtain the following bound for $I_1$:
    \begin{equation}\label{est_1}
        |I_1| \leq \frac{k_3}{d(x)} \int_0^{\frac{1}{2}d(x)} r^{-2s+1}\ dr \leq \frac{k_4}{d(x)^{2s-1}}, \text{ for } s\in \left[\frac{1}{2},1\right).
    \end{equation}

    To estimate $I_2$, we split the argument into two cases. When $s\in \left(\frac{1}{2},1\right)$, a direct calculation gives
    \begin{equation}\label{est_2}
        |I_2| \leq \|v\|_{C^1(\sr^N)} \int_{|z|>\frac{1}{2}d(x)} |z|^{-N-2s+1}\ dz \leq \frac{k_5}{d(x)^{2s-1}}.
    \end{equation}
    In the case $s= \frac{1}{2}$, since $\Om$ is a bounded domain and thanks to the construction of $\tilde{e}_0$ as in lemma \ref{extension of e}, we can choose $R>0$ large, independent of $x\in\Om_{\eta}$, such that whenever $|z|\geq R$ we have $v(x-z)=0$. Hence, in this case $I_2$ can be estimated as follows:
    \begin{align}\label{est_3}
        |I_2| &\leq \int_{\frac{1}{2}d(x) < |z| < R} \frac{|v(x) - v(x-z)|}{|z|^{N+1}}\ dz + \int_{|z|\geq R} \frac{|v(x)|}{|z|^{N+1}}\ dz\nonumber\\
        &\leq \|v\|_{C^1 (\sr^N)} \int_{\frac{1}{2}d(x) < |z| <R} |z|^{-N}\ dz + \|v\|_{L^{\infty}} \int_{|z|\geq R} |z|^{-N-1}\ dz \leq k_6.
    \end{align}
    Substituting \eqref{est_1}, \eqref{est_2} and \eqref{est_3} in \eqref{est_nonlocal}, the claim follows.

    Combining this estimate with \eqref{est_k1}, we thus obtain
    \begin{equation}\label{v_equn_second}
        -\Delta v + (-\De)^s v + \frac{k}{d(x)^{\ba +1}} v \leq - \frac{k_1}{d(x)^{2-\gamma}} + \frac{k_2}{d(x)^{2s-1}} \text{ in } \Om_{\eta}, \text{ for } s\in \left[\frac{1}{2},1 \right).
    \end{equation}
    Now choose $\gamma \in \left(1, \text{min} \left\{ 2-\beta, 3-2s \right\} \right)$ if $s\in \left(\frac{1}{2},1 \right)$ and $\gamma \in (1, 2- \ba)$ if $s= \frac{1}{2}$ so that for any $s\in \left[ \frac{1}{2}, 1 \right)$ we have
    \begin{equation}\label{v_equn_third}
        -\Delta v + (-\De)^s v + \frac{k}{d(x)^{\ba +1}} v \leq 0 \text{ in } \Om_{\eta},
    \end{equation}
    provided $\eta>0$ is small enough. The above equation is analogous to equation \eqref{v_equn} in the proof of lemma \ref{hopf_first}. Repeating a similar argument and applying the weak comparison principle, we conclude that $\frac{\pa w}{\pa \nu}<0$, or $\frac{\pa u_2}{\pa\nu} < \frac{\pa u_1}{\pa\nu}$ on $\pa\Om$. This completes the proof in the case $s\in \left[ \frac{1}{2},1 \right)$.
\end{proof}

\noindent\textbf{Proof of Theorem \ref{Hopf lemma_linear case}:} Bringing together lemmas \ref{strict comparison}, \ref{hopf_first} and \ref{hopf_second}, we conclude the proof of the Hopf-type Strong Comparison Principle.\hfill\qed

Now, we are in a position to present the proof of Theorem \ref{theorem_3}. While the central idea mirrors that of Theorem \ref{theorem_2}, we provide a concise outline for clarity and completeness, highlighting the key differences and modifications required to address the current scenario.

\noindent\textbf{Proof of Theorem \ref{theorem_3}:} We fix $\la\in (\la_*, \la^*)$ and define the sets $X$, $X_1$ and $X_2$ analogous to the proof of Theorem \ref{theorem_2}. Let $z_1$ be the maximal positive solution of \eqref{P_lambda_linear} in $X_1$ and $z_2$ be the minimal positive solution in $X_2$, obtained by the monotone iteration method.

In view of the boundary regularity result (Theorem 2.3, \cite{dhanya2024interiorboundaryregularitymixed}) for singular problems, the proof of Claim-1, that is, $T: X \rar X$ is completely continuous, proceeds analogously. Similarly, Claim-2 follows from the direct application of the Hopf-type SCP presented in Theorem \ref{Hopf lemma_linear case}.

As a result, the conditions for Proposition \ref{amann} are satisfied and we obtain a third fixed point $z_3 \in X \setminus (X_1 \cup X_2)$ of the solution map $T$. Therefore, if $\la\in (\la_*, \la^*)$, the equation \eqref{P_lambda_linear} has three distinct positive solutions.\hfill\qed


\section{Concluding Remark}

In essence, Theorem \ref{theorem_2} and Theorem \ref{theorem_3} establish three-solution theorems in two special cases: the non-singular case ($1<p<\infty$ and $\ba =0$) and the singular case with the linear operator ($p=2$ and $0<\ba <1$) respectively. However, in the general case when $p\not= 2$ and $0<\ba<1$, we could only obtain two positive solutions. Thus, the existence of a third solution for the following equation
\begin{equation}\label{conclusion_equation}
    \begin{aligned}
        -\p u + \ps u &= \la \frac{f(u)}{u^{\ba}} \text{ in } \Om, \text{ with } p\not= 2,\ 0<\beta<1 \\
        u >0 \text{ in } \Om,\ u &=0 \text{ in } \bdry
    \end{aligned}
\end{equation}
still remains an open question. We emphasize that the main difficulty in establishing the third solution is proving the Hopf-type Strong Comparison Principle for the nonlinear operator, as Theorem \ref{Hopf lemma_linear case} only addresses the linear case. Given the sub- and supersolutions developed in Section \ref{section-construction} and the recently established boundary regularity results (\cite{dhanya2024interiorboundaryregularitymixed}, \cite{bal2025regularityresultsclassmixed}), if the Hopf-type Strong Comparison Principle (Hopf-type SCP) for mixed local-nonlocal operators is available for singular problems, one can prove the three-solution theorem for the above equation \eqref{conclusion_equation}. The proof of the existence of a third solution will follow the arguments similar to those of Theorem \ref{theorem_3}.

In a forthcoming work, we develop a Hopf-type SCP that applies to nonlinear and nonhomogeneous mixed operators. Beyond the immediate application to a three-solution theorem, such a result carries intrinsic mathematical interest and has broader implications in the study of multiplicity results involving mixed operators.


\section*{Acknowledgements}

The author expresses gratitude to Dr. Dhanya R. for suggesting this problem and for the valuable discussions throughout the work. The author also sincerely thanks Prof. J. Giacomoni for his engaging feedback. The author was supported by the Prime Minister's Research Fellowship.


\bibliographystyle{amsplain}
\bibliography{reference}	

@article {giacomoni2019existence,
    AUTHOR = {Giacomoni, Jacques and Mukherjee, Tuhina and Sreenadh, Konijeti},
     TITLE = {Existence of three positive solutions for a nonlocal singular {D}irichlet boundary problem},
   JOURNAL = {Adv. Nonlinear Stud.},
  FJOURNAL = {Advanced Nonlinear Studies},
    VOLUME = {19},
      YEAR = {2019},
    NUMBER = {2},
     PAGES = {333--352},
      ISSN = {1536-1365,2169-0375},
   MRCLASS = {35R11 (35A15 35R09)},
  MRNUMBER = {3943307},
       DOI = {10.1515/ans-2018-0011}
}

@article {ko2011multiplicity,
    AUTHOR = {Ko, Eunkyung and Lee, Eun Kyoung and Shivaji, R.},
     TITLE = {Multiplicity results for classes of infinite positone
              problems},
   JOURNAL = {Z. Anal. Anwend.},
  FJOURNAL = {Zeitschrift f\"ur Analysis und ihre Anwendungen. Journal of
              Analysis and its Applications},
    VOLUME = {30},
      YEAR = {2011},
    NUMBER = {3},
     PAGES = {305--318},
      ISSN = {0232-2064,1661-4534},
   MRCLASS = {35J75 (35A01 35J25 35J57 35J92)},
  MRNUMBER = {2819497},
MRREVIEWER = {George\ Smyrlis},
       DOI = {10.4171/ZAA/1436}
}

@article {dhanya2015three,
    AUTHOR = {Dhanya, R. and Ko, Eunkyung and Shivaji, R.},
     TITLE = {A three solution theorem for singular nonlinear elliptic
              boundary value problems},
   JOURNAL = {J. Math. Anal. Appl.},
  FJOURNAL = {Journal of Mathematical Analysis and Applications},
    VOLUME = {424},
      YEAR = {2015},
    NUMBER = {1},
     PAGES = {598--612},
      ISSN = {0022-247X,1096-0813},
   MRCLASS = {35J75 (35A01 35A09 35J25 35J91)},
  MRNUMBER = {3286582},
MRREVIEWER = {Said\ El Manouni},
       DOI = {10.1016/j.jmaa.2014.11.012}
}

@article {arora2021multiplicity,
    AUTHOR = {Arora, Rakesh},
     TITLE = {Multiplicity results for nonhomogeneous elliptic equations
              with singular nonlinearities},
   JOURNAL = {Commun. Pure Appl. Anal.},
  FJOURNAL = {Communications on Pure and Applied Analysis},
    VOLUME = {21},
      YEAR = {2022},
    NUMBER = {6},
     PAGES = {2253--2269},
      ISSN = {1534-0392,1553-5258},
   MRCLASS = {35J75 (35J20)},
  MRNUMBER = {4414607},
       DOI = {10.3934/cpaa.2022056}
}

@article {acharya2021existence,
    AUTHOR = {Acharya, Ananta and Das, Ujjal and Shivaji, Ratnasingham},
     TITLE = {Existence and multiplicity results for {$p$}-{$q$}-{L}aplacian
              boundary value problems},
   JOURNAL = {Electron. J. Differential Equations},
  FJOURNAL = {Electronic Journal of Differential Equations},
    VOLUME = {Special Issue},
      YEAR = {2021},
    NUMBER = {1},
     PAGES = {293--300},
      ISSN = {1072-6691},
   MRCLASS = {35J70},
  MRNUMBER = {4494204},
       DOI = {10.58997/ejde.sp.01.a3}
}

@article {ramaswamy2004multiple,
    AUTHOR = {Ramaswamy, Mythily and Shivaji, Ratnasingham},
     TITLE = {Multiple positive solutions for classes of {$p$}-{L}aplacian
              equations},
   JOURNAL = {Differential Integral Equations},
  FJOURNAL = {Differential and Integral Equations. An International Journal
              for Theory \& Applications},
    VOLUME = {17},
      YEAR = {2004},
    NUMBER = {11-12},
     PAGES = {1255--1261},
      ISSN = {0893-4983},
   MRCLASS = {35J60},
  MRNUMBER = {2100025},
MRREVIEWER = {Gabriele\ Bonanno},
       DOI ={10.57262/die/1356060244}
}

@article {brown1981s,
    AUTHOR = {Brown, K. J. and Ibrahim, M. M. A. and Shivaji, R.},
     TITLE = {{$S$}-shaped bifurcation curves},
   JOURNAL = {Nonlinear Anal.},
  FJOURNAL = {Nonlinear Analysis. Theory, Methods \& Applications. An
              International Multidisciplinary Journal},
    VOLUME = {5},
      YEAR = {1981},
    NUMBER = {5},
     PAGES = {475--486},
      ISSN = {0362-546X,1873-5215},
   MRCLASS = {35B32 (35J65 58E07)},
  MRNUMBER = {613056},
MRREVIEWER = {M.\ Shearer},
       DOI = {10.1016/0362-546X(81)90096-1}
}

@article {antonini2025global,
    AUTHOR = {Antonini, Carlo Alberto and Cozzi, Matteo},
     TITLE = {Global gradient regularity and a {H}opf lemma for quasilinear
              operators of mixed local-nonlocal type},
   JOURNAL = {J. Differential Equations},
  FJOURNAL = {Journal of Differential Equations},
    VOLUME = {425},
      YEAR = {2025},
     PAGES = {342--382},
      ISSN = {0022-0396,1090-2732},
   MRCLASS = {35B65 (35J60 35M12 35R11)},
  MRNUMBER = {4853424},
       DOI = {10.1016/j.jde.2025.01.030}
}

@article {amann1976fixed,
    AUTHOR = {Amann, Herbert},
     TITLE = {Fixed point equations and nonlinear eigenvalue problems in
              ordered {B}anach spaces},
   JOURNAL = {SIAM Rev.},
  FJOURNAL = {SIAM Review. A Publication of the Society for Industrial and
              Applied Mathematics},
    VOLUME = {18},
      YEAR = {1976},
    NUMBER = {4},
     PAGES = {620--709},
      ISSN = {0036-1445},
   MRCLASS = {47H10 (58E15)},
  MRNUMBER = {415432},
MRREVIEWER = {D.\ H.\ Sattinger},
       DOI = {10.1137/1018114}
}

@article {de2024gradient,
    AUTHOR = {De Filippis, Cristiana and Mingione, Giuseppe},
     TITLE = {Gradient regularity in mixed local and nonlocal problems},
   JOURNAL = {Math. Ann.},
  FJOURNAL = {Mathematische Annalen},
    VOLUME = {388},
      YEAR = {2024},
    NUMBER = {1},
     PAGES = {261--328},
      ISSN = {0025-5831,1432-1807},
   MRCLASS = {49N60 (35J60 35R11 49J10)},
  MRNUMBER = {4693935},
       DOI = {10.1007/s00208-022-02512-7}
}

@article {garain2022mixed,
    AUTHOR = {Garain, Prashanta and Ukhlov, Alexander},
     TITLE = {Mixed local and nonlocal {S}obolev inequalities with extremal
              and associated quasilinear singular elliptic problems},
   JOURNAL = {Nonlinear Anal.},
  FJOURNAL = {Nonlinear Analysis. Theory, Methods \& Applications. An
              International Multidisciplinary Journal},
    VOLUME = {223},
      YEAR = {2022},
     PAGES = {Paper No. 113022, 35},
      ISSN = {0362-546X,1873-5215},
   MRCLASS = {35R11 (35A23 35J75 35J92)},
  MRNUMBER = {4444761},
       DOI = {10.1016/j.na.2022.113022}
}

@article {iannizzotto2023logistic,
    AUTHOR = {Iannizzotto, Antonio and Mosconi, Sunra and Papageorgiou,
              Nikolaos S.},
     TITLE = {On the logistic equation for the fractional {$p$}-{L}aplacian},
   JOURNAL = {Math. Nachr.},
  FJOURNAL = {Mathematische Nachrichten},
    VOLUME = {296},
      YEAR = {2023},
    NUMBER = {4},
     PAGES = {1451--1468},
      ISSN = {0025-584X,1522-2616},
   MRCLASS = {35R11},
  MRNUMBER = {4604546},
       DOI = {10.1002/mana.202100025}
}

@article {ko2024multiplicity,
    AUTHOR = {Ko, Eunkyung and Lee, Eun Kyoung and Shivaji, R.},
     TITLE = {Multiplicity of positive solutions to a class of
              {S}chr\"odinger-type singular problems},
   JOURNAL = {Discrete Contin. Dyn. Syst. Ser. S},
  FJOURNAL = {Discrete and Continuous Dynamical Systems. Series S},
    VOLUME = {17},
      YEAR = {2024},
    NUMBER = {5-6},
     PAGES = {2224--2233},
      ISSN = {1937-1632,1937-1179},
   MRCLASS = {35J25 (35J20 35J61 35J75)},
  MRNUMBER = {4762583},
       DOI = {10.3934/dcdss.2023166}
}

@article {dhanya2024interiorboundaryregularitymixed,
    AUTHOR = {Dhanya, R. and Giacomoni, Jacques and Jana, Ritabrata},
     TITLE = {Interior and boundary regularity of mixed local nonlocal
              problem with singular data and its applications},
   JOURNAL = {Nonlinear Anal.},
  FJOURNAL = {Nonlinear Analysis. Theory, Methods \& Applications. An
              International Multidisciplinary Journal},
    VOLUME = {262},
      YEAR = {2026},
     PAGES = {Paper No. 113940},
      ISSN = {0362-546X,1873-5215},
   MRCLASS = {35J60 (35B65 35J75)},
  MRNUMBER = {4957750},
       DOI = {10.1016/j.na.2025.113940},
       URL = {https://doi.org/10.1016/j.na.2025.113940},
}

@article {silvestre2007regularity,
    AUTHOR = {Silvestre, Luis},
     TITLE = {Regularity of the obstacle problem for a fractional power of
              the {L}aplace operator},
   JOURNAL = {Comm. Pure Appl. Math.},
  FJOURNAL = {Communications on Pure and Applied Mathematics},
    VOLUME = {60},
      YEAR = {2007},
    NUMBER = {1},
     PAGES = {67--112},
      ISSN = {0010-3640,1097-0312},
   MRCLASS = {35J05 (35B65 35R35 49N60)},
  MRNUMBER = {2270163},
MRREVIEWER = {Martin\ Fuchs},
       DOI = {10.1002/cpa.20153}
}

@article {biagi2024multiplicity,
    AUTHOR = {Biagi, Stefano and Vecchi, Eugenio},
     TITLE = {Multiplicity of positive solutions for mixed local-nonlocal
              singular critical problems},
   JOURNAL = {Calc. Var. Partial Differential Equations},
  FJOURNAL = {Calculus of Variations and Partial Differential Equations},
    VOLUME = {63},
      YEAR = {2024},
    NUMBER = {9},
     PAGES = {Paper No. 221, 45},
      ISSN = {0944-2669,1432-0835},
   MRCLASS = {35J75 (35B09 35B33 35D30 35M12)},
  MRNUMBER = {4802957},
       DOI = {10.1007/s00526-024-02819-0}
}

@article{bal2025regularityresultsclassmixed,
  title={Regularity results for a class of mixed local and nonlocal singular problems involving distance function},
  author={Bal, Kaushik and Das, Stuti},
  journal={ArXiv Preprint arXiv:2411.14217},
  pages={1--48},
  year={2025}
}

@article {arora2023combined,
    AUTHOR = {Arora, Rakesh and R{\u{a}}dulescu, Vicen{\c{t}}iu D.},
     TITLE = {Combined effects in mixed local-nonlocal stationary problems},
   JOURNAL = {Proc. Roy. Soc. Edinburgh Sect. A},
  FJOURNAL = {Proceedings of the Royal Society of Edinburgh. Section A.
              Mathematics},
    VOLUME = {155},
      YEAR = {2025},
    NUMBER = {1},
     PAGES = {10--56},
      ISSN = {0308-2105,1473-7124},
   MRCLASS = {35J75 (35B65 35R09 35R11 35S15 47G20)},
  MRNUMBER = {4899563},
       DOI = {10.1017/prm.2023.80}
}

@article {garain2023class,
    AUTHOR = {Garain, Prashanta},
     TITLE = {On a class of mixed local and nonlocal semilinear elliptic
              equation with singular nonlinearity},
   JOURNAL = {J. Geom. Anal.},
  FJOURNAL = {Journal of Geometric Analysis},
    VOLUME = {33},
      YEAR = {2023},
    NUMBER = {7},
     PAGES = {Paper No. 212, 20},
      ISSN = {1050-6926,1559-002X},
   MRCLASS = {35M10 (35B65 35J75 35R11)},
  MRNUMBER = {4578205},
       DOI = {10.1007/s12220-023-01262-5}
}

@article {biroud2023mixed,
    AUTHOR = {Biroud, Kheireddine},
     TITLE = {Mixed local and nonlocal equation with singular nonlinearity
              having variable exponent},
   JOURNAL = {J. Pseudo-Differ. Oper. Appl.},
  FJOURNAL = {Journal of Pseudo-Differential Operators and Applications},
    VOLUME = {14},
      YEAR = {2023},
    NUMBER = {1},
     PAGES = {Paper No. 13, 24},
      ISSN = {1662-9981,1662-999X},
   MRCLASS = {35R11 (35B65 35J75 35J92 46E35)},
  MRNUMBER = {4546847},
       DOI = {10.1007/s11868-023-00509-7}
}

@article {su2024multiple,
    AUTHOR = {Su, Xifeng and Valdinoci, Enrico and Wei, Yuanhong and Zhang,
              Jiwen},
     TITLE = {Multiple solutions for mixed local and nonlocal elliptic
              equations},
   JOURNAL = {Math. Z.},
  FJOURNAL = {Mathematische Zeitschrift},
    VOLUME = {308},
      YEAR = {2024},
    NUMBER = {3},
     PAGES = {Paper No. 40, 37},
      ISSN = {0025-5874,1432-1823},
   MRCLASS = {35J25 (35A09 35D30 35J61 35R09)},
  MRNUMBER = {4803021},
       DOI = {10.1007/s00209-024-03599-1}
}

@book {bebernes2013mathematical,
    AUTHOR = {Bebernes, Jerrold and Eberly, David},
     TITLE = {Mathematical problems from combustion theory},
    SERIES = {Applied Mathematical Sciences},
    VOLUME = {83},
 PUBLISHER = {Springer-Verlag, New York},
      YEAR = {1989},
     PAGES = {viii+177},
      ISBN = {0-387-97104-1},
   MRCLASS = {35Qxx (35K57 76N99 80A25)},
  MRNUMBER = {1012946},
MRREVIEWER = {P.\ Szeptycki},
       DOI = {10.1007/978-1-4612-4546-9}
}

@article {gelfand1963,
    AUTHOR = {Gel{$'$}fand, I. M.},
     TITLE = {Some problems in the theory of quasilinear equations},
   JOURNAL = {Amer. Math. Soc. Transl. (2)},
  FJOURNAL = {Amer. Math. Soc. Transl. (2)},
    VOLUME = {29},
      YEAR = {1963},
     PAGES = {295--381},
   MRCLASS = {35.00},
  MRNUMBER = {153960},
       DOI = {10.1090/trans2/029/12}
}

@article {dancer1980,
    AUTHOR = {Dancer, E. N.},
     TITLE = {On the structure of solutions of an equation in catalysis
              theory when a parameter is large},
   JOURNAL = {J. Differential Equations},
  FJOURNAL = {Journal of Differential Equations},
    VOLUME = {37},
      YEAR = {1980},
    NUMBER = {3},
     PAGES = {404--437},
      ISSN = {0022-0396,1090-2732},
   MRCLASS = {35B32 (58E07)},
  MRNUMBER = {590000},
MRREVIEWER = {P.\ C.\ Fife},
       DOI = {10.1016/0022-0396(80)90107-2}
}

@article {bratu1914equations,
    AUTHOR = {Bratu, G.},
     TITLE = {Sur les \'equations int\'egrales non lin\'eaires},
   JOURNAL = {Bull. Soc. Math. France},
  FJOURNAL = {Bulletin de la Soci\'et\'e{} Math\'ematique de France},
    VOLUME = {42},
      YEAR = {1914},
     PAGES = {113--142},
      ISSN = {0037-9484},
   MRCLASS = {99-04},
  MRNUMBER = {1504727},
       DOI = {10.24033/bsmf.943}
}

@article {joseph1972,
    AUTHOR = {Joseph, D. D. and Lundgren, T. S.},
     TITLE = {Quasilinear {D}irichlet problems driven by positive sources},
   JOURNAL = {Arch. Rational Mech. Anal.},
  FJOURNAL = {Archive for Rational Mechanics and Analysis},
    VOLUME = {49},
      YEAR = {1972/73},
     PAGES = {241--269},
      ISSN = {0003-9527},
   MRCLASS = {34B15},
  MRNUMBER = {340701},
MRREVIEWER = {Jean\ Mawhin},
       DOI = {10.1007/BF00250508}
}

@article {du2001conjecture,
    AUTHOR = {Du, Yihong and Lou, Yuan},
     TITLE = {Proof of a conjecture for the perturbed {G}elfand equation
              from combustion theory},
   JOURNAL = {J. Differential Equations},
  FJOURNAL = {Journal of Differential Equations},
    VOLUME = {173},
      YEAR = {2001},
    NUMBER = {2},
     PAGES = {213--230},
      ISSN = {0022-0396,1090-2732},
   MRCLASS = {35J65 (80A25)},
  MRNUMBER = {1834115},
MRREVIEWER = {Vitaly\ A.\ Volpert},
       DOI = {10.1006/jdeq.2000.3932}
}

@article {huang2015bifurcation,
    AUTHOR = {Huang, Shao-Yuan and Wang, Shin-Hwa},
     TITLE = {On {S}-shaped bifurcation curves for a two-point boundary
              value problem arising in a theory of thermal explosion},
   JOURNAL = {Discrete Contin. Dyn. Syst.},
  FJOURNAL = {Discrete and Continuous Dynamical Systems. Series A},
    VOLUME = {35},
      YEAR = {2015},
    NUMBER = {10},
     PAGES = {4839--4858},
      ISSN = {1078-0947,1553-5231},
   MRCLASS = {34B18 (34B08 34C23 74G35 80A20 80A25)},
  MRNUMBER = {3392652},
MRREVIEWER = {Ruyun\ Ma},
       DOI = {10.3934/dcds.2015.35.4839}
}

@article {huang2016one,
    AUTHOR = {Huang, Shao-Yuan and Wang, Shin-Hwa},
     TITLE = {Proof of a conjecture for the one-dimensional perturbed
              {G}elfand problem from combustion theory},
   JOURNAL = {Arch. Ration. Mech. Anal.},
  FJOURNAL = {Archive for Rational Mechanics and Analysis},
    VOLUME = {222},
      YEAR = {2016},
    NUMBER = {2},
     PAGES = {769--825},
      ISSN = {0003-9527,1432-0673},
   MRCLASS = {34B18 (34C23 80A25)},
  MRNUMBER = {3544317},
MRREVIEWER = {Jolanta\ Przybycin},
       DOI = {10.1007/s00205-016-1011-1}
}

@article {hung2011jde,
    AUTHOR = {Hung, Kuo-Chih and Wang, Shin-Hwa},
     TITLE = {A theorem on {S}-shaped bifurcation curve for a positone
              problem with convex-concave nonlinearity and its applications
              to the perturbed {G}elfand problem},
   JOURNAL = {J. Differential Equations},
  FJOURNAL = {Journal of Differential Equations},
    VOLUME = {251},
      YEAR = {2011},
    NUMBER = {2},
     PAGES = {223--237},
      ISSN = {0022-0396,1090-2732},
   MRCLASS = {34B08 (34B18 34C23 47J15)},
  MRNUMBER = {2800152},
       DOI = {10.1016/j.jde.2011.03.017}
}

@article {wang2008p,
    AUTHOR = {Wang, Shin-Hwa and Yeh, Tzung-Shin},
     TITLE = {Exact multiplicity of solutions and {S}-shaped bifurcation
              curves for the {$p$}-{L}aplacian perturbed {G}elfand problem
              in one space variable},
   JOURNAL = {J. Math. Anal. Appl.},
  FJOURNAL = {Journal of Mathematical Analysis and Applications},
    VOLUME = {342},
      YEAR = {2008},
    NUMBER = {2},
     PAGES = {1175--1191},
      ISSN = {0022-247X,1096-0813},
   MRCLASS = {34B15 (34B18 34C23 35J60 47J15)},
  MRNUMBER = {2445267},
       DOI = {10.1016/j.jmaa.2007.12.026}
}

@article {charro2023infinity,
    AUTHOR = {Charro, Fernando and Son, Byungjae and Wang, Peiyong},
     TITLE = {The {G}elfand problem for the infinity {L}aplacian},
   JOURNAL = {Math. Eng.},
  FJOURNAL = {Mathematics in Engineering},
    VOLUME = {5},
      YEAR = {2023},
    NUMBER = {2},
     PAGES = {Paper No. 022, 28},
      ISSN = {2640-3501},
   MRCLASS = {35J94},
  MRNUMBER = {4411361},
MRREVIEWER = {Davide\ Buoso},
       DOI = {10.3934/mine.2023022}
}

@article {huang2025minkowski,
    AUTHOR = {Huang, Shao-Yuan and Wang, Shin-Hwa},
     TITLE = {Bifurcation curves for the one-dimensional perturbed {G}elfand
              problem with the {M}inkowski-curvature operator},
   JOURNAL = {J. Differential Equations},
  FJOURNAL = {Journal of Differential Equations},
    VOLUME = {416},
      YEAR = {2025},
     PAGES = {700--726},
      ISSN = {0022-0396,1090-2732},
   MRCLASS = {34B18 (34C23 35J93 74G35)},
  MRNUMBER = {4808808},
       DOI = {10.1016/j.jde.2024.10.002}
}

@article {molino2016singular,
    AUTHOR = {Molino, Alexis},
     TITLE = {Gelfand type problem for singular quadratic quasilinear
              equations},
   JOURNAL = {NoDEA Nonlinear Differential Equations Appl.},
  FJOURNAL = {NoDEA. Nonlinear Differential Equations and Applications},
    VOLUME = {23},
      YEAR = {2016},
    NUMBER = {5},
     PAGES = {Art. 56, 20},
      ISSN = {1021-9722,1420-9004},
   MRCLASS = {35J62 (35A01 35B09 35J25 35J75)},
  MRNUMBER = {3549895},
MRREVIEWER = {Alan\ V.\ Lair},
       DOI = {10.1007/s00030-016-0409-7}
}

@article {rosoton2014fractional,
    AUTHOR = {Ros-Oton, Xavier},
     TITLE = {Regularity for the fractional {G}elfand problem up to
              dimension 7},
   JOURNAL = {J. Math. Anal. Appl.},
  FJOURNAL = {Journal of Mathematical Analysis and Applications},
    VOLUME = {419},
      YEAR = {2014},
    NUMBER = {1},
     PAGES = {10--19},
      ISSN = {0022-247X,1096-0813},
   MRCLASS = {35R11 (35B65)},
  MRNUMBER = {3217131},
MRREVIEWER = {Changpin\ Li},
       DOI = {10.1016/j.jmaa.2014.04.048}
}

@article {molino2025random,
    AUTHOR = {Mazon, J. M. and Molino, A. and Toledo, J.},
     TITLE = {Gelfand-type problems in random walk spaces},
   JOURNAL = {Calc. Var. Partial Differential Equations},
  FJOURNAL = {Calculus of Variations and Partial Differential Equations},
    VOLUME = {64},
      YEAR = {2025},
    NUMBER = {5},
     PAGES = {Paper No. 168, 46},
      ISSN = {0944-2669,1432-0835},
   MRCLASS = {35R02 (05C81 35A15 35J60)},
  MRNUMBER = {4913053},
       DOI = {10.1007/s00526-025-03021-6}
}

@article{parks1961criticality,
    author = {Parks, John R.},
    title = {Criticality criteria for various configurations of a self‐heating chemical as functions of activation energy and temperature of assembly},
    journal = {The Journal of Chemical Physics},
    volume = {34},
    number = {1},
    pages = {46--50},
    year = {1961},
    issn = {0021-9606},
    doi = {10.1063/1.1731612}
}

@article {sattinger1975nonlinear,
    AUTHOR = {Sattinger, D. H.},
     TITLE = {A nonlinear parabolic system in the theory of combustion},
   JOURNAL = {Quart. Appl. Math.},
  FJOURNAL = {Quarterly of Applied Mathematics},
    VOLUME = {33},
      YEAR = {1975/76},
     PAGES = {47--61},
      ISSN = {0033-569X,1552-4485},
   MRCLASS = {35B20 (35K60)},
  MRNUMBER = {463631},
MRREVIEWER = {F.\ Hoppensteadt},
       DOI = {10.1090/qam/463631}
}

@article {parter1974,
    AUTHOR = {Parter, Seymour V.},
     TITLE = {Solutions of a differential equation arising in chemical
              reactor processes},
   JOURNAL = {SIAM J. Appl. Math.},
  FJOURNAL = {SIAM Journal on Applied Mathematics},
    VOLUME = {26},
      YEAR = {1974},
     PAGES = {687--716},
      ISSN = {0036-1399},
   MRCLASS = {34B15},
  MRNUMBER = {382772},
MRREVIEWER = {Yiannis\ Sficas},
       DOI = {10.1137/0126063}
}

@article {tam1979,
    AUTHOR = {Tam, K. K.},
     TITLE = {Construction of upper and lower solutions for a problem in
              combustion theory},
   JOURNAL = {J. Math. Anal. Appl.},
  FJOURNAL = {Journal of Mathematical Analysis and Applications},
    VOLUME = {69},
      YEAR = {1979},
    NUMBER = {1},
     PAGES = {131--145},
      ISSN = {0022-247X},
   MRCLASS = {80A25 (35Q20 65P05)},
  MRNUMBER = {535286},
       DOI = {10.1016/0022-247X(79)90183-5}
}

@article {kapila1979,
    AUTHOR = {Kapila, A. K. and Matkowsky, B. J.},
     TITLE = {Reactive-diffuse systems with {A}rrhenius kinetics: multiple
              solutions, ignition and extinction},
   JOURNAL = {SIAM J. Appl. Math.},
  FJOURNAL = {SIAM Journal on Applied Mathematics},
    VOLUME = {36},
      YEAR = {1979},
    NUMBER = {2},
     PAGES = {373--389},
      ISSN = {0036-1399},
   MRCLASS = {80A30},
  MRNUMBER = {524507},
MRREVIEWER = {Richard\ J.\ Field},
       DOI = {10.1137/0136028}
}

\end{document}